\documentclass[article]{siamart1116}
\usepackage[utf8]{inputenc}
\usepackage{amssymb}
\usepackage{amsmath}
\usepackage{color}
\usepackage{tikz}
\usetikzlibrary{arrows}
\usetikzlibrary{arrows.meta}
\usetikzlibrary{patterns}
\usepackage{booktabs}
\usepackage{threeparttable}
\newsiamremark{remark}{Remark} 
\newsiamremark{asmp}{Assumption} 

\newcommand{\scalOmj}[1]{\left( #1 \right)_{\Omega_{m_j}}}
\newcommand{\scalOmjT}[1]{\left( #1 \right)_{\Omega_{m_j}^{T}}}
\newcommand{\scalOfT}[1]{\left( #1 \right)_{\Omega_f^T}}
\newcommand{\scalOf}[1]{\left( #1 \right)_{\Omega_f}}
\newcommand{\pcut}[1]{\left[#1\right]_+}
\newcommand{\ncut}[1]{\left[#1\right]_-}
\definecolor{color1}{rgb}{0, 0, 1}
\definecolor{color2}{rgb}{0.28235, 0.81961, 0.8}
\definecolor{color3}{rgb}{0.6, 0.19608, 0.8}
\definecolor{color4}{rgb}{1, 0.078431, 0.57647}
\definecolor{color5}{rgb}{0.64706, 0.16471, 0.16471}
\definecolor{color6}{rgb}{0.86275, 0.078431, 0.23529}
\definecolor{color7}{rgb}{1, 0.84314, 0}
\definecolor{color7star}{rgb}{1, 0.64706, 0}
\definecolor{color8}{rgb}{0, 0.50196, 0}
\definecolor{color9}{rgb}{0.18039, 0.5451, 0.34118}

\title{Rigorous upscaling of unsaturated flow in fractured porous media %
}
	
\author{Florian List %
\thanks {Department of Mathematics, University of Hasselt, Belgium}
\and
Kundan Kumar 
\thanks{Department of Mathematics, University of Bergen, Norway}
\and
Iuliu Sorin Pop 
\footnotemark[1]
$^{,}$
\footnotemark[2]
\and
Florin A. Radu %
\footnotemark[2]
}

\headers{Upscaling of unsaturated flow in fractured porous media}{F. List, K. Kumar, I. S. Pop, F. A. Radu}

\date{today}
\begin{document}
\maketitle
\begin{abstract}
In this work, we consider a mathematical model for flow in a unsaturated porous medium containing a fracture. In all subdomains (the fracture and the adjacent matrix blocks) the flow is governed by Richards' equation. The submodels are coupled by physical transmission conditions expressing the continuity of the normal fluxes and of the pressures. We start by analyzing the case of a fracture having a fixed width-length ratio, called $\varepsilon > 0$. Then we take the limit $\varepsilon \to 0$ and give a rigorous proof for the convergence towards effective models. This is done in different regimes, depending on how the ratio of porosities and permeabilities in the fracture, respectively matrix scale with respect to $\varepsilon$, and leads to a variety of effective models. Numerical simulations confirm the theoretical upscaling results.
\end{abstract}


\begin{AMS}
{ 35B27, 35A35,  35J25, 35K65,}
\end{AMS}

\section{Introduction}	
Fractured porous media arise in a multitude of environmental and technical applications, including fragmented rocks, hydraulic fracturing, carbon dioxide sequestration, and geothermal systems. Fractures are thin formations, in which the hydraulic properties such as porosity and permeability differ significantly from those of the surrounding matrix blocks. Hence, fractures have a crucial impact on fluid flow \cite{AdlerThovert}, and fractures or entire fracture networks must be incorporated in the mathematical models for fluid flow. This is challenging from the numerical point of view, firstly due to the high geometrical complexity of fracture networks and secondly because grid cells with a high aspect ratio or a very fine grid resolution within the fractures and in the adjacent matrix region are required.  

\par In order to overcome the latter difficulty, it is appealing to embed fractures as lower-dimensional manifolds into a higher-dimensional domain (e.g. as lines in a two-dimensional domain) and thus to reduce the fracture width. Depending on the context, fractures may block or conduct fluid flow, which can be expressed for example by coupling the mathematical model for the matrix blocks with a differential equation on the lower-dimensional fractures. Herein, we prove that such models result naturally from models with positive fracture width in the limit case where the width passes to zero. The presented model in this work provides a physically-consistent fundation for discrete fracture modeling approaches (e.g. \cite{Durlofsky, Hadi}).

\par We consider a two-dimensional model for unsaturated fluid flow in a fractured porous medium. For the ease of presentation, the geometry is given by two rectangular matrix blocks, separated by a single fracture. Here we assume that, next to the matrix blocks, the fracture is a porous medium too, as encountered e.g. in the case of sediment-filled fractures \cite{HoughEtAl}, or layered porous media \cite{NeuweilerEichel}. We assume that the pore space of the porous medium is filled with a liquid (say, water) and air. Provided that the domain is interconnected and connected to the surface, the assumption that the air is infinitely mobile is justified, and the air pressure can be set to zero in the full two-phase model. In this way, the governing model in the matrix blocks and in the fracture is the Richards equation \cite{Richards}, 
\begin{equation}
\partial_t (\phi S (\psi)) - \nabla \cdot (K_a K(S(\psi)) \nabla \psi) = f. 
\label{eq:richards}
\end{equation}
Here, $\psi$ denotes the pressure head, $\phi$ the porosity of the medium, $S$ the water saturation, $K_a$ and $K$ stand for the absolute, respective relative hydraulic conductivity, and $f$ is a source or sink term. 
For simplicity, the gravity is neglected, and the absolute permeability is a scalar, but all the results in this paper can be extended to include gravity, or anisotropic media.

\par In \eqref{eq:richards}, $\phi$ and $K_a$ are medium-dependent parameters. Similarly, the water saturation $S$ is a given, increasing function of $\psi$, whereas the relative conductivity $K$ is a given function of $S$. As for $\phi$ and $K_a$, these relationships depend on the type of the material in the porous medium. Therefore, all these material properties may be different in the fracture and in the matrix blocks, see e.g. \cite{Helmig}. Well-known are the van Genuchten--Mualem \cite{vanGenuchten} and Brooks--Corey \cite{BrooksCorey} relationships. 

The Richards equation is a non-linear parabolic partial differential equation and may degenerate wherever the flow is saturated $S'(\psi) = 0$ (the fast diffusion case) or $K(S(\psi)) \to 0$ (the slow diffusion case). However, the rigorous mathematical results in this work only cover non-degenerate cases, when the medium is strictly unsaturated. On the other hand, the effective models derived here remain formally valid also in the degenerate cases. 

In view of its practical relevance, the Richards equation has been investigated thoroughly in the mathematical literature. Without being exhaustive, we mention \cite{AltLuckhaus,AltEtAl,vanDuijnPeletier} for results concerning the existence of weak solutions including degenerate cases. Uniqueness results are obtained in e.g. \cite{Otto1, Otto2}. 
The numerical methods are developed in agreement with the analytical results. One remarkable feature is that when compared to the case of the heat equation, the solutions to the Richards equation lack regularity. For this reason, as well as for ensuring stability, the implicit Euler scheme is commonly used for the time discretisation. The outcome is a sequence of time discrete nonlinear elliptic equations, which are generally solved by means of linear iterative schemes like Newton, fixed-point or Picard. Such methods are discussed and compared e.g. in \cite{ListRadu}. For the spatial discretisation, we mention \cite{EymardEtAl1999, EymardEtAl2006, KlausenEtAl}  where finite volume approaches are presented,  \cite{ArbogastEtAl1996,  RaduEtAl2004, PopLin2004, RaduNumMath2008, WoodwardDawson} for mixed finite element methods, and \cite{Ebmeyer, NochettoVerdi} for finite element schemes.   

In all papers mentioned above, the parameters and nonlinearities are either fixed over the entire domain, or vary smoothly. In other words, the problems can be considered over the entire domain, without paying particular attention to the fact that there are different media involved. In the present work, the medium consists of different homogeneous blocks, connected through transmission conditions that will be given below. In this context, domain decomposition methods represent an efficient way to reduce both the problem complexity, and to deal with the occurrence of different homogeneous blocks. We refer to \cite{BerningerDiss} for a domain decomposition scheme applied to unsaturated flows in layered solis, and to \cite{seus} for a scheme combining linearization and domain decomposition techniques in each iteration. 

The present work is considering a particular situation, where the medium consists of two homogeneous blocks, separated by a thin, homogeneous structure, the fracture. We consider a two-dimensional situation, and let $\varepsilon >0$ be a dimensionless number giving the ratio between the fracture width and length. Since the fracture is assumed thin, $\varepsilon$ can be seen as a small parameter. If fractures are viewed as two-dimensional objects, their discretization becomes complex as the mesh should either contain anisotropic elements respecting the fracture shape, or should be extremely fine. To avoid such issues, one possibility is to approximate fractures as lower dimensional elements in the entire domain. This implies finding appropriate, reduced dimensional models for the fracture, and how these are connected to the models in the matrix blocks. In this sense, we mention \cite{Andersen, Jaffre}, where the reduced dimensional models for two-phase flow, respective reactive transport in fractured media are derived by formal arguments based on a transversal averaging of the model inside fractures. Similar results, but using anisotropic asymptotic expansion methods in terms of $\varepsilon$ are obtained in \cite{vanDuijnPop, Showalter1, Showalter2, PopBogersKumar2016}, where the convergence of the averaging process is proved rigorously when $\varepsilon \searrow 0$. We also refer to \cite{Brenner, Boon, Alessio, Alessio1}, where reduced dimensional models for flow in fractured media are presented with emphasis on developing appropriate numerical schemes. 

The models considered here are assuming that the pressure is continuous at the interfaces separating the matrix blocks and the fracture. In other words, entry pressure models leading to the extended pressure condition derived in \cite{vanDuijndeNeef} are disregarded. For such models we also mention that homogenization results are obtained in \cite{popvanduijntrapping, Henning, schweizer, Adam}. In particular, oil trapping effects are well explained by such models. 

Although the pressure is assumed continuous at the interfaces separating the homogeneous blocks, this does not rule out the situation where the averaged pressure across a fracture may still become discontinuous in the reduced dimensional models. Such models are discussed e.g. in \cite{Brenner}. The present analysis does not cover such cases, but we refer to \cite{ListThesis} for the formal derivation of such models in the specific context discussed here. 

Still referring to fractured media, but with a different motivation, are the works in \cite{wick} for a phase field model describing the propagation of fluid filled fractures and \cite{Girault} for iterative approaches to static fractures.


\par The main goal in this work is to give a mathematically rigorous derivation of the reduced dimensional models in fractured media. Based on anisotropic asymptotic expansion, we give rigorous proofs for the convergence of averaging procedure when passing $\varepsilon$, the ratio between the fracture width and length, to zero. Depending on how the ratio of the porosities and of the absolute permeabilities in the different types of materials scale w.r.t. $\varepsilon$, five different reduced dimensional models are obtained. More precisely, if the fracture is more permeable than the adjacent blocks, it becomes a preferential flow path. On the contrary, if the fracture is less permeable than the blocks, the fluid will have a preference to flow in the blocks. In consequence, the reduced dimensional fracture equation for the fracture can be an interface condition or a differential equation. Such results are obtained by means of a formal derivation in \cite{Ahmed, Jaffre}. Our approach is in spirit of \cite{Tunc, Showalter1, Showalter2}, where the single phase flow through a highly permeable fracture is considered. This corresponds to a particular choice of the scalings in the porosity, respectively absolute permeability ratio. The key mathematical challenge in this work is to obtain estimates specifying the explicit dependence on $\varepsilon$. The starting mathematical equations have coefficients that depend on $\varepsilon$ and the major part of the work is in identifying the dependence of the estimates on  $\varepsilon$. 

\par The outline of this work is as follows. In Section \ref{sec:model}, the coupled model is formulated, and a non-dimensionalisation procedure is carried out in order to derive a dimensionless model, which is then used for the upscaling. Two scaling parameters, $\kappa$ and $\lambda$, are introduced. These account for the scaling of the porosities and absolute hydraulic conductivities with respect to $\varepsilon$. In Section \ref{sec:main_result}, we briefly state the main results of this work.  Section \ref{sec:existence} is concerned with the existence of solutions to the model for a constant but positive fracture width, i.e. $\varepsilon > 0$. This is done by applying Rothe's method (see e.g. \cite{Kacur}).  Based on compactness arguments we prove the existence of solutions  to the coupled model. Further, to reduce the dimensionality of the fracture, we investigate the limit of vanishing fracture width, that is $\varepsilon \to 0$ in Section \ref{sec:upscaling}. Section \ref{sec:simulations} presents numerical simulations that confirm our theoretical upscaling results. 

%

\section{Model}
\label{sec:model}
First, we formulate the model in dimensional form. Thereafter, we introduce reference quantities and make assumptions on their scaling with respect to one another. This is where the scaling parameters $\kappa$ and $\lambda$ come into play. By relating the dimensional quantities to the reference quantities, the non-dimensional model is derived, which will be considered in the subsequent sections. 
\subsection{Dimensional model}
We resort to a simple two-dimensional geometry consisting of two square solid matrix blocks with edge length $L$ separated by a fracture of width $l$. The geometry is illustrated in Figure \ref{fig:Figure_dimensions} (left).
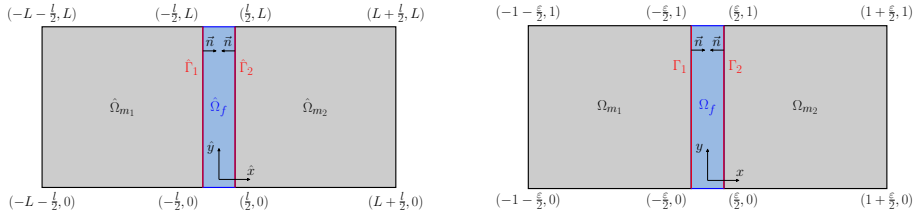
\begin{figure}[!h]
\begin{minipage}{0.49\textwidth}
\begin{center}
\resizebox{0.9\textwidth}{!}{%
\definecolor{myblue}{rgb}{0.6,0.73,0.89}
\begin{tikzpicture}[
pile/.style={thick, ->, >=stealth', shorten <=0pt, shorten
    >=2pt}]
{\Huge
\draw [line width = 2pt, fill=black!20] (-11,0) rectangle (-1,10);
\draw [line width = 2pt, fill=black!20] (1,0) rectangle (11,10);
\draw [blue, line width = 2pt, fill=myblue] (-1,0) rectangle (1,10);
\node [below](a) at (-11,0) {$(-L-\frac{l}{2},0)$};
\node [below left](b1) at (-1,0) {$(-\frac{l}{2},0)$};
\node [below](b) at (-1,0) {};
\node [below right](c1) at (1,0) {$(\frac{l}{2},0)$};
\node [below](c) at (1,0) {};
\node [below](d) at (11,0) {$(L+\frac{l}{2},0)$};
\node [above](e) at (-11,10) {$(-L-\frac{l}{2},L)$};
\node [above left](f1) at (-1,10) {$(-\frac{l}{2},L)$};
\node [above](f) at (-1,10) {};
\node [above right](g1) at (1,10) {$(\frac{l}{2},L)$};
\node [above](g) at (1,10) {};
\node [above](h) at (11,10) {$(L+\frac{l}{2},L)$};
\node(i) at (-6,5) {$\hat{\Omega}_{m_1}$};
\node(j) at (6,5) {$\hat{\Omega}_{m_2}$};
\node(k) [blue] at (0,5) {$\hat{\Omega}_f$};
\draw [red, line width = 2pt] (b) -- (f);
\draw [red, line width = 2pt] (c) -- (g);
\node(l) [red, left] at (-1,7.5) {$\hat{\Gamma}_1$};
\node(m) [red, right] at (1,7.5) {$\hat{\Gamma}_2$};
\draw[pile] (0.0,0.5) -- (0.0,2.5) node [pos = 1, left] {$\hat{y}$};
\draw[pile] (0.0,0.5) -- (2.0,0.5) node [pos = 1, above] {$\hat{x}$};
\draw[pile] (-1.0,8.5) -- (-0.1,8.5) node[pos = 0.5, above] {$\vec{n}$};
\draw[pile] (1.0,8.5) -- (0.1,8.5) node[pos = 0.5, above] {$\vec{n}$};
}
\end{tikzpicture}
}
\end{center}
\end{minipage}
\begin{minipage}{0.49\textwidth}
\begin{center}
\resizebox{0.9\textwidth}{!}{%
\definecolor{myblue}{rgb}{0.6,0.73,0.89}
\begin{tikzpicture}[
pile/.style={thick, ->, >=stealth', shorten <=0pt, shorten
    >=2pt}]
{\Huge
\draw [line width = 2pt, fill=black!20] (-11,0) rectangle (-1,10);
\draw [line width = 2pt, fill=black!20] (1,0) rectangle (11,10);
\draw [blue, line width = 2pt, fill=myblue] (-1,0) rectangle (1,10);
\node [below](a) at (-11,0) {$(-1-\frac{\varepsilon}{2},0)$};
\node [below left](b1) at (-1,0) {$(-\frac{\varepsilon}{2},0)$};
\node [below](b) at (-1,0) {};
\node [below right](c1) at (1,0) {$(\frac{\varepsilon}{2},0)$};
\node [below](c) at (1,0) {};
\node [below](d) at (11,0) {$(1+\frac{\varepsilon}{2},0)$};
\node [above](e) at (-11,10) {$(-1-\frac{\varepsilon}{2},1)$};
\node [above left](f1) at (-1,10) {$(-\frac{\varepsilon}{2},1)$};
\node [above](f) at (-1,10) {};
\node [above right](g1) at (1,10) {$(\frac{\varepsilon}{2},1)$};
\node [above](g) at (1,10) {};
\node [above](h) at (11,10) {$(1+\frac{\varepsilon}{2},1)$};
\node(i) at (-6,5) {$\Omega_{m_1}$};
\node(j) at (6,5) {$\Omega_{m_2}$};
\node(k) [blue] at (0,5) {$\Omega_f$};
\draw [red, line width = 2pt] (b) -- (f);
\draw [red, line width = 2pt] (c) -- (g);
\node(l) [red, left] at (-1,7.5) {$\Gamma_1$};
\node(m) [red, right] at (1,7.5) {$\Gamma_2$};
\draw[pile] (0.0,0.5) -- (0.0,2.5) node [pos = 1, left] {$y$};
\draw[pile] (0.0,0.5) -- (2.0,0.5) node [pos = 1, above] {$x$};
\draw[pile] (-1.0,8.5) -- (-0.1,8.5) node[pos = 0.5, above] {$\vec{n}$};
\draw[pile] (1.0,8.5) -- (0.1,8.5) node[pos = 0.5, above] {$\vec{n}$};
}
\end{tikzpicture}
}
\end{center}
\label{fig:Figure}
\end{minipage}
\caption{Dimensional (left) and dimensionless (right) geometry of the fracture and the surrounding matrix blocks}
\label{fig:Figure_dimensions}
\end{figure} \\
The subscripts $m$ and $f$ indicate the matrix blocks and the fracture, respectively. They are defined as
\begin{equation}
\begin{aligned}
\hat{\Omega}_{m_1} &:= \left(-\frac L 2 - \frac l 2, - \frac l 2\right) \times (0,L), &\qquad
\hat{\Gamma}_1 &:= \left\lbrace -\frac l 2 \right\rbrace \times (0,L), &\\
\hat{\Omega}_{m_2} &:= \left(\frac l 2, \frac L 2 + \frac l 2\right) \times (0,L), &\qquad
\hat{\Gamma}_2 &:= \left\lbrace \frac l 2 \right\rbrace \times (0,L), &\\
\hat{\Omega}_f &:= \left(-\frac l 2, \frac l 2\right) \times (0,L). &
\end{aligned}
\end{equation}
We use superscript hats for denoting quantities associated with the dimensional model in order to distinguish them from the dimensionless quantities which will be introduced subsequently. \\
The model defined on the dimensional geometry is given by\\[1ex]
\text{Problem } $\mathcal{P}_D$:
\begin{equation*}
\left\lbrace
\begin{aligned}
\partial_{\hat{t}} (\phi_m \hat{S}_m(\hat{\psi}_{m_j})) + \hat{\nabla} \cdot \hat{v}_{m_j} &= \hat{f}_{m_j}& \qquad &\text{in } \hat{\Omega}_{m_j}^{\hat{T}}, \\
\hat{v}_{m_j} &= - \hat{K}_{a, m}\hat{K}_m(\hat{S}_m(\hat{\psi}_{m_j})) \hat{\nabla} \hat{\psi}_{m_j}& \qquad &\text{in } \hat{\Omega}_{m_j}^{\hat{T}}, \\
\partial_{\hat{t}} (\phi_f \hat{S}_f(\hat{\psi}_f)) + \hat{\nabla} \cdot \hat{v}_f &= \hat{f}_f& \qquad &\text{in } \hat{\Omega}_f^{\hat{T}}, \\
\hat{v}_f &= - \hat{K}_{a, f} \hat{K}_f(\hat{S}_f(\hat{\psi}_f)) \hat{\nabla} \hat{\psi}_f& \qquad &\text{in } \hat{\Omega}_f^{\hat{T}}, \\
\hat{v}_{m_j} \cdot \vec{n} &= \hat{v}_f \cdot \vec{n}& \qquad &\text{on } \hat{\Gamma}_j^{\hat{T}}, \\
\hat{\psi}_{m_j} &= \hat{\psi}_f& \qquad &\text{on } \hat{\Gamma}_j^{\hat{T}}, \\
\hat{\psi}_\rho(0) &= \hat{\psi}_{\rho,I}& \qquad &\text{in } \hat{\Omega}_\rho,
\end{aligned}
\right.
\end{equation*}
for $\rho \in \lbrace m_1, m_2, f \rbrace$, $j \in \lbrace 1, 2 \rbrace $, where $\hat{\Omega}_m := \hat{\Omega}_{m_1} \cup \hat{\Omega}_{m_2}$, and where we set $\Omega^{\hat{T}} := \Omega \times (0,\hat{T}]$ for all spatial domains $\Omega$ and a given final time $\hat{T} > 0$. Furthermore, $\vec{n}$ is a normal vector pointing from $\hat{\Omega}_{m_j}$ into $\hat{\Omega}_f$. $\hat{S}_\rho$, $\hat{v}_\rho$, $\hat{\psi}_\rho$, $\hat{K}_{a, \rho}$, $\hat{K}_\rho$ are the saturation, flux, pressure height,  absolute and relative hydraulic conductivities in the subdomain $\Omega_\rho$, respectively.  $\hat{\psi}_{\rho,I}$ is a given initial condition and $\hat{f}_\rho$ is a source/sink term. 
\\
In words, the Richards equation is modeling the flow in the fracture and in the matrix blocks, supplemented with the continuity of the normal flux and of the pressure as transmission conditions, and initial conditions. 
\subsection{Non-dimensionalisation}
We define $\varepsilon := \frac{l}{L},$ that is, the ratio of the fracture width to its length.   We take $L$ as the reference length scale. The dimensionless geometry is as shown in Figure \ref{fig:Figure} (right):
\begin{equation}
\begin{aligned}
\Omega_{m_1} &= \left(-\frac 1 2 - \frac \varepsilon 2, - \frac \varepsilon 2\right) \times (0,1), &\qquad
\Gamma_1 &= \left\lbrace -\frac \varepsilon 2 \right\rbrace \times (0,1), &\\
\Omega_{m_2} &= \left(\frac \varepsilon 2, \frac 1 2 + \frac \varepsilon 2\right) \times (0,1), &\qquad
\Gamma_2 &= \left\lbrace \frac \varepsilon 2 \right\rbrace \times (0,1), &\\
\hat{\Omega}_f &= \left(-\frac \varepsilon 2, \frac \varepsilon 2\right) \times (0,1). &
\end{aligned}
\end{equation}
Since the pressure is continuous at the interfaces, we define a single reference pressure head for the entire domain, $\bar{\psi} = L$. We further assume that the matrix blocks have the same properties. Consequently, only two absolute hydraulic conductivities are encountered, $\bar{K}_m$ and $\bar{K}_f$, respectively. As reference time scale we set
\begin{equation}
\bar{T} := \frac{\phi_m L^2}{\hat{K}_m \bar{\psi}} = \frac{\phi_m L}{\hat{K}_m}.
\label{eq:bar_T}
\end{equation}
The dimensionless pressure heads are then given as $\psi_{m_j} = \hat{\psi}_{m_j}/L$ and $\psi_f = \hat{\psi}_f/L$, the dimensionless time as $t = \hat{t}/\bar{T}$, and the final time as $T = \hat{T}/\bar{T}$.  As regards the source terms, we set $f_{m_j} = \hat{f}_{m_j} \bar{T}/\phi_m$ and $f_{f} = \hat{f}_{f} \bar{T}/\phi_m$. 

The functions $\hat{S}_\rho$ and $\hat{K}_\rho$ (where $\rho \in \lbrace m_1, m_2, f \rbrace$) are dimensionless. Expressed in terms of dimensionless arguments, they become ${S}_\rho$ and ${K}_\rho$. 
\\
Using the Darcy law in the mass balance equation in Problem $\mathcal{P}_D$ and using \eqref{eq:bar_T} one gets the dimensionless equations for the matrix blocks ($j \in \{1, 2\}$)
\begin{equation}
 \partial_t S_m(\psi_{m_j}) - \nabla \cdot \left(K_m(S_m(\psi_{m_j})) \nabla \psi_{m_j} \right) = f_{m_j} \quad \text{in } \Omega_{m_j}^T.
\end{equation}

As announced in the introduction, the results depend on the types of materials in the blocks and in the fractures. More exactly, important is how the ratio of the porosities and of the absolute hydraulic conductivities in the fracture and the matrix blocks are scaling w.r.t. $\varepsilon$, 
\begin{equation}
\frac{\phi_f}{\phi_m} \varpropto \varepsilon^\kappa, \quad \text{and} \quad \frac{\hat{K}_{a, f}}{\hat{K}_{a, m}} \varpropto \varepsilon^\lambda. 
\label{eq:scaling}
\end{equation}
Here $\kappa, \lambda \in \mathbb{R}$ are scaling parameters. For the ease of notation, we take the constants of proportionality to be one for the analysis.
Using this and applying the same ideas as above, the model in the fracture becomes 
\begin{equation}
\partial_t (\varepsilon^\kappa S_f(\psi_{f})) - \nabla \cdot \left(\varepsilon^\lambda K_f(S_f(\psi_{f})) \nabla \psi_{f} \right) = f_{f} \quad \text{in } \Omega_{f}^T.
\end{equation}
The transmission condition for the normal flux transforms into
\begin{equation}
K_m(S_m(\psi_{m_j})) \nabla \psi_{m_j} \cdot \vec{n} = \varepsilon^\lambda \, K_f(S_f(\psi_{f})) \nabla \psi_{f} \cdot \vec{n} \qquad \text{on } \Gamma_j^T.
\end{equation}
In what follows, the dimensionless fracture width $\varepsilon > 0$ is a model parameter. Given $\varepsilon > 0$, the dimensionless model becomes
\begin{equation*}
\text{Problem } \mathcal{P}_\varepsilon:
\left\lbrace
\begin{aligned}
\partial_t S_m(\psi^\varepsilon_{m_j}) + \nabla \cdot v^\varepsilon_{m_j} &= f_{m_j}& \qquad &\text{in } \Omega_{m_j}^T, \\
v^\varepsilon_{m_j} &= - K_m(S_m(\psi^\varepsilon_{m_j})) \nabla \psi^\varepsilon_{m_j}& \qquad &\text{in } \Omega_{m_j}^T, \\
\partial_t(\varepsilon^\kappa S_f(\psi^\varepsilon_f)) + \nabla \cdot v^\varepsilon_f &= f_f^\varepsilon& \qquad &\text{in } \Omega_f^T, \\
v^\varepsilon_f &= - \varepsilon^\lambda K_f(S_f(\psi^\varepsilon_f)) \nabla \psi^\varepsilon_f& \qquad &\text{in } \Omega_f^T, \\
v^\varepsilon_{m_j} \cdot \vec{n} &= v^\varepsilon_f \cdot \vec{n}& \qquad &\text{on } \Gamma_j^T, \\
\psi^\varepsilon_{m_j} &= \psi^\varepsilon_f& \qquad &\text{on } \Gamma_j^T, \\
\psi^\varepsilon_\rho(0) &= \psi_{\rho,I}& \qquad &\text{in } \Omega_\rho.
\end{aligned}
\right.
\end{equation*}
\begin{remark}[Scaling parameters]
The scaling parameters $\kappa, \lambda \in \mathbb{R}$ will be crucial in determining the effective models in the limit $\varepsilon \to 0$. $\kappa$ is related to the storage capacity of the fracture: for $\kappa < 0$, the reference porosity of the fracture increases for decreasing $\varepsilon$ as compared to the reference porosity of the matrix blocks. For $\kappa \leq -1$, the fracture maintains its ability to store water as $\varepsilon$ approaches zero. For $\kappa = 0$, no scaling occurs, and for $\kappa > 0$, the storage ability of the fracture decreases for $\varepsilon \to 0$ due to the decline of both the fracture volume (assuming fixed $L$) and of the fracture porosity. 

\par The parameter $\lambda$ instead gives the scaling of the conductivities. Here we consider the case $\lambda <1$. $\lambda < 0$ corresponds to the case of a highly conductive fracture when compared to the matrix, which means that the flow through the fracture is more rapid. Whenever $\lambda > 0$ the fractures are less permeable than the blocks. The case $\lambda = 0$ means comparable conductivities. The case   $\lambda \geq 1$ corresponds to impermeable fractures, leading in the limit $\varepsilon \rightarrow 0$ to models where the pressures at the matrix block at each side of the fractures are discontinuous (see \cite{ListThesis}).  To analyze such cases rigorously, one can employ techniques that are similar to ones used in \cite{Maria}, which are different from those used here. Accordingly, we only restrict to the case when $\lambda < 1$. 

\end{remark}
\section{Main result}
\label{sec:main_result}
Our main result is the rigorous derivation of effective models replacing the fracture by an interface. Table \ref{table:model_summary} provides a brief summary of the effective models for the entire range $(\kappa, \lambda) \in [-1,\infty) \times (-\infty, 1) $ except for the case when $\kappa = -1, \lambda \in (-1,1)$. Due to the non-linearity of the time derivative term involved, the identification of the limit requires stronger estimates than we have. Therefore, this case is left unresolved. 
\scriptsize{
\begin{center}
\begin{threeparttable}
\begin{tabular}{@{}lllc@{}}
\toprule
                    & Fracture equation                            						 & Parameter range for $(\kappa, \lambda)$ \\
\midrule 
{\color{color1} Effective model I}  			 & Richards' equation                                 & $\lbrace -1 \rbrace \times \lbrace -1 \rbrace$ \\
{\color{color2} Effective model II} 			 & Elliptic equation                           		  & $(-1, \infty) \times \lbrace -1 \rbrace$ \\
{\color{color3} Effective model III } & ODE for spatially constant pressure                & $\lbrace -1 \rbrace \times (-\infty, -1)$ \\
{\color{color4} Effective model IV \tnote{\color{black} }}  & Spatially constant pressure                        & $(-1, \infty) \times (-\infty, -1)$             				         \\
{\color{color5}Effective model V}			& Pressure and flux continuity between matrix blocks         					 &  $(-1, \infty) \times ( -1, 1)$ \\

\bottomrule
\end{tabular}

\vspace{0.1cm}
\caption{Summary of the effective models}
\label{table:model_summary}
\end{threeparttable}
\end{center}
}
\normalsize
We state the strong formulation of the effective models for $(\kappa,\lambda) \in [-1,\infty) \times (-\infty, 1)$ except $\kappa = -1, \lambda \in (-1,1)$.
\subsection{Effective models: strong formulation}
Effective model I consists of Richards' equation in the matrix block subdomains and the one-dimensional Richards' equation in the fracture. It occurs for $\kappa = \lambda = -1$. \\

\par {\color{color1}\text{Effective model I}}:
\begin{equation}
\left\lbrace
\begin{aligned}
\partial_t S_m(\Psi_{m_j}) - \nabla \cdot \left(K_m(S_m(\Psi_{m_j})) \nabla \psi_{m_j}\right) &= f_{m_j},& \qquad &\text{in } \Omega_{m_j}^T, \\
\partial_t S_f(\Psi_f) - \partial_y \left(K_f(S_f(\Psi_f)) \partial_y \Psi_f\right) &= \left[\vec{q}_m\right]_\Gamma, &\qquad &\text{on } \Gamma^T, \\
\Psi_{m_j} &= \Psi_f, &\qquad &\text{on } \Gamma^T, \\
\Psi_{m_j}(0) &= \Psi_{m_j,I}, & \qquad &\text{on } \Omega_{m_j}, \\
\Psi_{f}(0)   &= \Psi_{f,I},  & \qquad  &\text{on } \Gamma,
\end{aligned}
\right.
\end{equation}
where 
\begin{equation}	
\left[\vec{q}_m\right]_\Gamma := \left(K_m(S_m(\Psi_{m_1}))\nabla \Psi_{m_1} \cdot \vec{n}_{m_1} + K_m(S_m(\Psi_{m_2}))\nabla \Psi_{m_2} \cdot \vec{n}_{m_2}\right)\big{|}_\Gamma
\end{equation} 
is the flux difference between the two solid matrix subdomains acting as a source term for Richards' equation in the fracture (note that $\vec{n}_{m_1} = -\vec{n}_{m_2})$. \\
If the porosity ratio does not change with vanishing fracture width and the permeability ratio is taken to be reciprocally proportional to the fracture width, i.e. $\kappa > -1$ and $\lambda = -1$, one ends up with an effective model consisting of  Richards' equation in the matrix blocks and a stationary elliptic equation in the fracture.\\

\par{\color{color2}\text{Effective model II}}:
\begin{equation}
\left\lbrace
\begin{aligned}
\partial_t S_m(\Psi_{m_j}) - \nabla \cdot \left(K_m(S_m(\Psi_{m_j})) \nabla \Psi_{m_j}\right) &= f_{m_j}, \qquad &\text{in } \Omega_{m_j}^T, \\
-\partial_y \left(K_f(S_f(\Psi_f)) \partial_y \Psi_f\right) &= \left[\vec{q}_m\right]_\Gamma, \qquad &\text{on } \Gamma^T, \\
\Psi_{m_j} &= \Psi_f, \qquad &\text{on } \Gamma^{T}, \\
\Psi_{m_j}(0) &= \Psi_{m_j,I}, &\text{on } \Omega_{m_j}.
\end{aligned}
\right.
\end{equation}
For $\kappa = -1$ and $\lambda < -1$, the pressure in the fracture becomes spatially constant in the effective model, and due to the pressure continuity, acts as the boundary condition for  the pressure in the matrices.\\

\par{\color{color3}\text{Effective model III}}:
\begin{equation}
\left\lbrace
\begin{aligned}
\partial_t S_m(\Psi_{m_j}) - \nabla \cdot \left(K_m(S_m(\Psi_{m_j})) \nabla \Psi_{m_j}\right) &= f_{m_j}, \hspace{2.5cm} \text{in } \Omega_{m_j}^T, \\
\Psi_f(t,y) &= \Psi_f(t), \hspace{2.3cm} \text{on } \Gamma^T, \\
\partial_t S_f(\Psi_f)(t) &= \int_0^1 \left[\vec{q}_m\right]_\Gamma \ dy, \hspace{1.3cm} \text{on } \Gamma^T, \\
\Psi_{m_j} &= \Psi_f, \hspace{2.75cm} \text{on } \Gamma^{T}, \\
\Psi_{m_j}(0) &= \Psi_{m_j,I}, \hspace{2.2cm} \text{on } \Omega_{m_j}, \\
\Psi_{f}(0)   &= \Psi_{f,I}, \hspace{2.7cm} \text{on } \Gamma.
\end{aligned}
\right.
\end{equation}
For $\kappa > -1$ and $\lambda < -1$, the pressure in the fracture takes a constant value at each time, in such a way that the total flux across the fracture is conserved:\\

\par{\color{color4}\text{Effective model IV}}:
\begin{equation}
\left\lbrace
\begin{aligned}
\partial_t S_m(\Psi_{m_j}) - \nabla \cdot \left(K_m(S_m(\Psi_{m_j})) \nabla \Psi_{m_j}\right) &= f_{m_j}, \qquad &\text{in } \Omega_{m_j}^T, \\
\Psi_f(t,y) &= \Psi_f(t), \qquad &\text{on } \Gamma^T, \\
\int_0^1 \left[\vec{q}_m\right]_\Gamma \ dy &= 0, \qquad &\text{on } \Gamma^T, \\
\Psi_{m_j} &= \Psi_f, \qquad &\text{on } \Gamma^{T}, \\
\Psi_{m_j}(0) &= \Psi_{m_j,I}, &\text{on } \Omega_{m_j}.
\end{aligned}
\right.
\end{equation}

For $\kappa > -1$ and $\lambda \in (-1,1)$, an effective model results in which the fracture as a physical entity has disappeared. In this case, both the pressure and the flux are continuous on $\Gamma$.\\

\par{\color{color5}\text{Effective model V}}:
\begin{equation}
\left\lbrace
\begin{aligned}
\partial_t S_m(\psi_{m_j}) - \nabla \cdot \left(K_m(S_m(\psi_{m_j})) \nabla \psi_{m_j}\right) &= f_{m_j}, \qquad &\text{in } \Omega_{m_j}^T, \\
\left[\vec{q}_m\right]_\Gamma &= 0, \qquad &\text{on } \Gamma^T, \\
\psi_{m_j} &= \psi_f, \qquad &\text{on } \Gamma^{T}, \\
\psi_{m_j}(0) &= \psi_{m_j,I}, &\text{on } \Omega_{m_j}.
\end{aligned}
\right.
\end{equation}

\section{Existence}
\label{sec:existence}
This section is concerned with the existence of a (weak) solution to Problem $\mathcal{P}_\varepsilon$ for a fixed fracture width $\varepsilon > 0$. We proceed in the spirit of \cite{PopBogersKumar2016}, where a linear model for reactive flow  with nonlinear transmission conditions at the interfaces is considered. For the sake of readability, we drop the superscript $\varepsilon$ since it is fixed throughout this section.
\subsection{Notation}
In this work, we use common notation from functional analysis. The space $L^2(\Omega)$ contains all real valued square integrable functions on a domain $\Omega \subset \mathbb{R}^d$, and $W^{1,2}(\Omega) \subset L^2(\Omega)$ stands for the subset of functions whose weak first order derivatives lie in $L^2(\Omega)$ as well. Furthermore, Bochner spaces $L^2(0,T;X)$ will be used, where $X$ stands for a Banach space. For all domains $\Omega \subset \mathbb{R}^d$ and time intervals $[0,T]$, we introduce the following abbreviations for the norm and the inner product:
\begin{equation}
\begin{aligned}
\|\cdot\|_\Omega &:= \|\cdot\|_{L^2(\Omega)},& \quad &\text{ and }& \quad \|\cdot\|_{\Omega^T} &:= \|\cdot\|_{L^2(0,T;L^2(\Omega))}, \\
\left(\cdot,\cdot\right)_\Omega &:= \left(\cdot,\cdot\right)_{L^2(\Omega)},&  \quad &\text{ and }& \qquad \left(\cdot,\cdot\right)_{\Omega^T} &:=  \left(\cdot,\cdot\right)_{L^2(0,T;L^2(\Omega))}.
\end{aligned}
\end{equation}
In view of the particular Problem $\mathcal{P}_\varepsilon$, we use the following conventions:
\begin{itemize}
\item $\rho$ is the index for the subdomain and takes values in $\lbrace m_1, m_2, f \rbrace$,
\item $j$ is the index for specifying the matrix block subdomain and takes values in $\lbrace 1, 2 \rbrace$,
\item for functions $g$ which are the same in both matrix block subdomains (such as $S$, $K$, $\ldots$), we define $g_{m_1} = g_{m_2} =: g_m$, which allows to write e.g. $S_\rho(\psi_\rho)$.
\end{itemize}
Moreover, $C \geq 0$ is a generic constant. \\
\subsection{Assumptions}
For the analysis, we assume that the following conditions are satisfied:
\begin{itemize}
\item[$(A_f)$] $f_\rho \in C(0,T;L^2(\Omega_\rho))$ and there exists $M_{f} > 0$ such that $|f_\rho| \leq M_f$ a.e. in $\Omega_\rho^T$.
\item[$(A_{D_S})$] $S_\rho \in C^1(\mathbb{R})$.
\item[$(A_S)$] There exist $m_S, M_{S} > 0$ such that $0 < m_S \leq S_\rho'(\psi_\rho) \leq M_{S}$ for all $\psi_\rho \in \mathbb{R}$.
\item[$(A_{D_K})$] $K_\rho \in C^1(\mathbb{R})$ and $K_\rho'(S_\rho) > 0$ for all $S_\rho \in \mathbb{R}$.
\item[$(A_K)$] There exist $m_{K}, M_{K} > 0$ such that $0 < m_K \leq K_\rho(S_\rho) < M_K $ for all $S_\rho \in \mathbb{R}$.
\item[$(A_\rho)$] There exists $M_\rho > 0$ such that $|\psi_{\rho,I}| \leq M_\rho$ a.e. in $\Omega_\rho$.
\end{itemize}
\begin{remark}[Assumptions]
 Note that due to Assumption $(A_S)$, we only consider the regular parabolic case here.
Assumption $(A_K)$ excludes the slow diffusion case and guarantees the existence of a minimum positive permeability everywhere.
Moreover, for the sake of presentation, we make the assumption $S_\rho(0) = 0$, which can easily be achieved by redefining ${S}_\rho(\psi) = \tilde{S}_\rho(\psi) - \tilde{S}_\rho(0)$. Assumption $(A_S)$ immediately yields the estimate $\|S(\psi)\|_{\Omega} \leq M_S \|\psi\|_{\Omega}$.
\end{remark}
\subsection{Weak solution}
We establish a suitable notion of a solution to Problem $\mathcal{P}_\varepsilon$. For this purpose, we define the function spaces
{
\begin{equation}
\begin{aligned}
\mathcal{V}_{m_j} &\subset  W^{1,2}(\Omega_{m_j}), \\
\mathcal{V}_f &\subset W^{1,2}(\Omega_f),
\end{aligned}
\end{equation} 
}
where the desired boundary conditions are implicitly imposed by the choice of the subspaces $\mathcal{V}_{m_j}$ and $\mathcal{V}_f$. We choose homogeneous Dirichlet conditions for the external boundaries in this section, that is, 
$\mathcal{V}_{m_j} = \lbrace u \in W^{1,2}(\Omega_{m_j}): u = 0 \text{ on } \partial \Omega_{m_j} \setminus \Gamma_j \rbrace$. This choice of boundary condition simplifies the presentation and extensions to other boundary conditions such as no flow Neumann conditions can be made without additional difficulties.
Note that these spaces depend on the fixed fracture width $\varepsilon$. 

The weak formulation of Problem $\mathcal{P}_\varepsilon$ reads as follows:
\begin{definition}[Weak solution]
\label{def:weak_solution}
A triple $(\psi_{m_1}, \psi_{m_2}, \psi_f)$ belonging to  product space $ L^2(0,T;\mathcal{V}_{m_1}) \times L^2(0,T;\mathcal{V}_{m_2}) \times L^2(0,T;\mathcal{V}_f)$ is called a weak solution to Problem $\mathcal{P}_\varepsilon$ if 
\begin{equation}
\begin{aligned}
\psi_{m_1} = \psi_f \ \text{ on } \Gamma_1 \quad \text{and} \quad \psi_{m_2} = \psi_f \ \text{ on } \Gamma_2 \quad \text{for a.e. } t \in [0,T],
\end{aligned}
\end{equation}
in the sense of traces, and
\begin{equation}
\begin{aligned}
&\hspace{.3cm} - \sum_{j=1}^2 \scalOmjT{S_m(\psi_{m_j}),\partial_t \phi_{m_j}} - \varepsilon^\kappa \scalOfT{S_f(\psi_f), \partial_t \phi_f} \\
&\hspace{.3cm} + \sum_{j=1}^2 \scalOmjT{K_m(S_m(\psi_{m_j}))\nabla \psi_{m_j}, \nabla \phi_{m_j}} + \varepsilon^\lambda \scalOfT{K_f(S_f(\psi_f))\nabla \psi_f, \nabla \phi_f} \\
&= \sum_{j=1}^2 \scalOmjT{f_{m_j}, \phi_{m_j}} + \scalOfT{f_f, \phi_f} \\
&\hspace{.3cm} + \sum_{j=1}^2 \left(S_m(\psi_{m_j,I}), \phi_{m_j}(0)\right)_{\Omega_{m_j}} + \varepsilon^\kappa \left(S_f(\psi_{f,I}), \phi_{f}(0)\right)_{\Omega_{f}},
\label{eq:weak_solution}
\end{aligned}
\end{equation}
for all $(\phi_{m_1}, \phi_{m_2}, \phi_f) \in W^{1,2}(0,T;\mathcal{V}_{m_1}) \times W^{1,2}(0,T;\mathcal{V}_{m_2}) \times W^{1,2}(0,T;\mathcal{V}_f)$ satisfying 
\begin{equation}
\phi_{m_1} = \phi_f \ \text{ on } \Gamma_1 \quad \text{and} \quad \phi_{m_2} = \phi_f \ \text{ on } \Gamma_2 \quad \text{for a.e. } t \in [0,T],
\end{equation}
and 
\begin{equation}
\phi_\rho(T) = 0, \qquad \text{for } \rho \in \lbrace m_1, m_2, f \rbrace.
\end{equation}
\end{definition}
Note that it makes sense to evaluate the test functions $\phi_\rho$ at the times $t = 0$ and $t = T$ in the above definition since the space $W^{1,2}(0,T;\mathcal{V}_\rho)$ is embedded in $C(0,T;\mathcal{V}_\rho)$.
\subsection{Time discretisation}
In what follows, we discretise the problem in time using an implicit Euler approach, which gives elliptic equations at every discrete time $t_k = k \Delta t$, for $k \in \lbrace 0, \ldots, N \rbrace$, where $N \in \mathbb{N}$. We assume without loss of generality that $N \Delta t = T$. Here, $\Delta t > 0$ denotes the fixed time step size. Choose $\psi_\rho^0 = \psi_{\rho,I}$ and let the sequence of solutions in domain $\Omega_\rho$ of the time-discrete problems be given as $\lbrace \psi_\rho^k \rbrace$. Moreover, let $f_\rho^k := f_\rho(t_k)$. The definition of a weak solution to the time-discrete problem is given by
\begin{definition}
\label{def:sol_discrete_problem}
Let $k > 0$ and let $(\psi_{m_1}^{k-1}, \psi_{m_2}^{k-1}, \psi_f^{k-1}) \in \mathcal{V}_{m_1} \times \mathcal{V}_{m_1} \times \mathcal{V}_{f}$ be given. We call $(\psi_{m_1}^k, \psi_{m_2}^k, \psi_{f}^k) \in \mathcal{V}_{m_1} \times \mathcal{V}_{m_1} \times \mathcal{V}_{f}$ a weak solution to the time-discrete problem at time $t_k$ if it satisfies 
\begin{equation}
\psi_{m_1}^k = \psi_f^k \ \text{ on } \Gamma_1 \quad \text{and} \quad \psi_{m_2}^k = \psi_f^k \ \text{ on } \Gamma_2
\end{equation}
in the sense of traces, and
\begin{equation}
\begin{aligned}
&\hspace{0.1cm} \sum_{j=1}^2 \scalOmj{S_m(\psi_{m_j}^k), \phi_{m_j}} + \varepsilon^\kappa \scalOf{S_f(\psi_{f}^k), \phi_{f}} \\
&\hspace{0.1cm} + \Delta t \sum_{j=1}^2 \scalOmj{K_m(S_m(\psi_{m_j}^k)) \nabla \psi_{m_j}^k, \nabla \phi_{m_j}} + \varepsilon^\lambda \Delta t \scalOf{K_f(S_f(\psi_{f}^k)) \nabla \psi_{f}^k, \nabla \phi_{f}} \\
&= \Delta t \sum_{j=1}^2 \scalOmj{f_{m_j}^k, \phi_{m_j}} + \Delta t \scalOf{f_{f}^k, \phi_{f}} \\
&\hspace{0.1cm}+ \sum_{j=1}^2 \scalOmj{S_m(\psi_{m_j}^{k-1}), \phi_{m_j}} + \varepsilon^\kappa \scalOf{S_f(\psi_{f}^{k-1}), \phi_{f}},
\end{aligned}
\label{eq:discrete_solution}
\end{equation}
for all $(\phi_{m_1}, \phi_{m_2}, \phi_f) \in \mathcal{V}_{m_1} \times \mathcal{V}_{m_2} \times \mathcal{V}_f$ satisfying $\phi_{m_j} = \phi_f$ on $\Gamma_j$ for $j \in \lbrace 1, 2 \rbrace$. 
\end{definition}
\subsection{Existence of solution for the time-discrete problem}
We begin with the existence of solution for the time-discrete problem as given in Definition \ref{def:sol_discrete_problem}. We show that the solution triple $(\psi_{m_1}^k, \psi_{m_2}^k, \psi_{f}^k) \in \mathcal{V}_{m_1} \times \mathcal{V}_{m_2} \times \mathcal{V}_{f}$ satisfying Definition \ref{def:sol_discrete_problem} can be interpreted as a solution to an elliptic problem having coefficients with possibly jump discontinuities.  The existence of solution is thus tantamount to showing that of an elliptic problem defined in the whole domain having possibly discontinuous coefficients. The latter follows from standard elliptic theory. We start by introducing the space $\mathcal{V}$
\begin{align*}
\mathcal{V} := \left \{  (\psi_{m_1}, \psi_{m_2}, \psi_{f})    \in \mathcal{V}_{m_1} \times \mathcal{V}_{m_2} \times \mathcal{V}_{f}, \text{ s.t. } \psi_{m_1} =  \psi_f  \text{ at } \Gamma_1,\;  \psi_{m_2} =  \psi_f \text{ at } \Gamma_2\right \},
\end{align*} 
equipped with the norm
\[ \left \| \psi \right \|_{\mathcal{V}} := \sqrt {\left ( \| \psi_{m_1}\|_{W^{1,2}(\Omega_{m_1})}^2 + \|\psi_{m_2}\|_{W^{1,2}(\Omega_{m_2})}^2 + \| \psi_{f}\|_{W^{1,2}(\Omega_f)}^2 \right )}.\] As before, the equalities on the interfaces $\Gamma_1, \Gamma_2$ are in the sense of traces. 
Below we will use the characteristic function  $\chi_\rho$ of $\Omega_\rho, \rho \in \lbrace m_1, m_2, f \rbrace$ in defining a function over $\Omega$ given a triple in $\mathcal{V}$. We have the following proposition showing that $\mathcal{V}$ is isomorphic to $W^{1,2}(\Omega)$.   
\begin{proposition} \label{prop:H1}
Given $\psi \in W^{1,2}(\Omega)$, its restriction to $\Omega_\rho$, $\rho \in \lbrace m_1, m_2, f \rbrace$, defines a triple $\left (\psi_{m_1}, \psi_{m_2}, \psi_{f} \right ) \in \mathcal{V}$. Conversely, given $\left (\psi_{m_1}, \psi_{m_2}, \psi_{f} \right ) \in \mathcal{V}$, $\psi = \sum_\rho \psi_\rho \chi_\rho, \rho \in \lbrace m_1, m_2, f\rbrace$ lies in $W^{1,2}(\Omega)$.
\end{proposition}
\begin{proof}
We start with the first part. For smooth functions, the assertion is obvious. Using a density argument and trace inequalities on $\Gamma_1$ and $\Gamma_2$, the extension to $W^{1,2}$ functions is straightforward. For the converse, the boundedness of the $L^2$ norm is clear. Further, it is sufficient to prove that the weak derivatives of the triple $(\psi_{m_1}, \psi_{m_2}, \psi_{f}) \in \mathcal{V}$ are equal to those of $\psi$ restricted to $\Omega_\rho$. For a given triple $(\psi_{m_1}, \psi_{m_2}, \psi_{f}) \in \mathcal{V}$, let $\psi = \sum_\rho \psi_\rho \chi_\rho, \rho \in \lbrace m_1, m_2, f \rbrace$. Let $\phi$ be weak derivative of $\psi$ in the $i-$th direction. Using partial integration, for any smooth function $w$ with compact support in $\Omega$, 
\[ \int_{\Omega} \phi w dx =  -  \int_{\Omega}  \psi \partial_i w dx = -  \sum_\rho \int_{\Omega_\rho}  \psi_\rho \partial_i w dx. \]
Using partial integration on each subdomain  
\[ -  \sum_\rho \int_{\Omega_\rho}  \psi_\rho \partial_i w dx =  \sum_\rho \int_{\Omega_\rho} \partial_i  \psi_\rho w dx, \]
 where the terms on the boundaries $\Gamma_1, \Gamma_2$ get cancelled due to traces being equal. The last equality shows that $\phi$ restricted to $\Omega_\rho, \rho \in \lbrace m_1, m_2, f \rbrace$ is equal to the weak derivative of $\psi$ in the $i$-th  direction. This proves the proposition. 
\end{proof}
\begin{remark}
With respect to the norm $ \| \psi\|_{W^{1,2}(\Omega)} = \sqrt { \|\psi\|_\Omega^2 + \| \nabla \psi \|_\Omega^2}$, and the same  for the $W^{1,2}$ norm on $\Omega_\rho, \rho \in \lbrace m_1, m_2, f \rbrace$, the isomorphism of $\mathcal{V}$ to $W^{1,2}(\Omega)$ is an isometry. 
\end{remark}

Next, we consider an elliptic problem defined in the entire domain $\Omega$. For a given triple $(\psi_{m_1}^{k-1}, \psi_{m_2}^{k-1}, \psi_f^{k-1}) \in \mathcal{V}$,  define $\psi^{k-1} = \sum_\rho \psi_\rho^{k-1} \chi_\rho, \rho \in \lbrace m_1, m_2, f \rbrace,$ and the coefficients $K = K_{m_1} \chi_{m_1} +\varepsilon^\lambda K_{f} \chi_{f} + K_{m_2} \chi_{m_2} ,$ and $S = S_{m_1} \chi_{m_1} +\varepsilon^\kappa S_{f} \chi_{f} + S_{m_2} \chi_{m_2}$. Definition of a solution for Problem P$_\Omega$ is as follows:  

\begin{definition}[Weak solution of Problem $\mathcal{P}_\Omega$] \label{def:fulldomain}
Given $\psi^{k-1}$, a weak solution $\psi^k \in W^{1,2}_0 (\Omega)$ is such that for all $\phi \in W^{1,2}_0 (\Omega)$ it holds that 
\begin{equation}
\begin{aligned}
 \left ({S(\psi^k), \phi} \right )_\Omega + \Delta t \left ( K(S(\psi^k)) \nabla \psi^k, \nabla \phi \right )_\Omega = \Delta t  \left ( f^k, \phi \right )_\Omega + \left ( S(\psi^{k-1}), \phi \right )_\Omega.
\end{aligned}
\label{eq:discrete_solution_fulldomain}
\end{equation}
\end{definition}
The above problem therefore is a non-linear elliptic problem with positive elliptic coefficient and a lower order reaction term that is monotone with respect to unknown and piecewise smooth functions with respect to space. The existence of solution in the Hilbert space $W^{1,2}_0(\Omega)$ is standard and can be read from \cite{BoccardoMurat2, BoccardoMurat}.  This is stated in the next lemma. 

\begin{lemma} \label{lemma:pomega}
There exists a weak solution of problem P$_\Omega$ in the sense of Definition \ref{def:fulldomain}.
\end{lemma}

The summary of the above discussion results in the existence of a solution for time discrete problem as per Definition \ref{def:sol_discrete_problem} and is given below. 
\begin{lemma} \label{lemma:existence_timediscrete}
Given $(\psi_{m_1}^{k-1}, \psi_{m_2}^{k-1}, \psi_f^{k-1}) \in \mathcal{V}_{m_1} \times \mathcal{V}_{m_1} \times \mathcal{V}_{f}$,  $k > 0$ , there exists a solution triple $(\psi_{m_1}^k, \psi_{m_2}^k, \psi_{f}^k) \in \mathcal{V}_{m_1} \times \mathcal{V}_{m_1} \times \mathcal{V}_{f}$ and 
\begin{equation}
\psi_{m_1}^k = \psi_f^k \ \text{ on } \Gamma_1 \quad \text{and} \quad \psi_{m_2}^k = \psi_f^k \ \text{ on } \Gamma_2.
\end{equation}
\end{lemma}
\begin{proof}
The existence result in Lemma \ref{lemma:pomega} provides $\psi^k \in W^{1,2}_0(\Omega)$. Proposition \ref{prop:H1} gives a triple $(\psi_{m_1}^k, \psi_{m_2}^k, \psi_{f}^k) \in \mathcal{V}_{m_1} \times \mathcal{V}_{m_1} \times \mathcal{V}_{f}$ satisfying $\psi_{m_1}^k = \psi_f^k \ \text{ on } \Gamma_1 $ and  $\psi_{m_2}^k = \psi_f^k \ \text{ on } \Gamma_2$. Moreover, Proposition \ref{prop:H1} states the equality of weak derivatives of $\psi^k$ restricted to $\Omega_\rho$ with those of $\psi_\rho^k$. Starting from \eqref{eq:discrete_solution_fulldomain}, this yields the existence result in Lemma \ref{lemma:existence_timediscrete}. 
\end{proof}

Our interface conditions on $\Gamma_1, \Gamma_2$ are natural: the continuity of flux and the pressures. In case when the interface conditions are nonlinear, we refer to the work of \cite{cances,JaegerKutev,JaegerSimon, PopBogersKumar2016}. 

\subsection{{\it A priori} estimates}
We define in each domain the energy functional
\begin{equation}
\mathcal{W}_\rho(\psi_\rho) = \int_0^{\psi_\rho} S_\rho'(\varphi) \, \varphi \ d\varphi,
\end{equation}
which we will require in the proof of the following {\it a priori} estimate. First, we gather some properties of $\mathcal{W}_\rho$ in a simple lemma, which is based on Assumption $(A_S)$:
\begin{lemma}
\label{lemma:W_properties}
The functional $\mathcal{W}_\rho$ satisfies the following inequalities:
\begin{equation}
\begin{aligned}
\mathcal{W}_\rho(\psi_\rho) &\geq 0, \\
\mathcal{W}_\rho(\psi_\rho)-\mathcal{W}_\rho(\xi_\rho) &\leq \psi_\rho (S_\rho(\psi_\rho)-S_\rho(\xi_\rho)), \\
m_S \frac{\psi_\rho^2}{2} \leq \mathcal{W}_\rho(\psi_\rho) &\leq M_S \frac{\psi_\rho^2}{2},
\end{aligned}
\label{eq:W_properties}
\end{equation}
for all $\psi_\rho, \xi_\rho \in \mathbb{R}$.
\end{lemma}
We obtain the following estimate for the time-discrete solution:
\begin{lemma}[A priori estimate I]
\label{lemma:a_priori_estimate_discrete}
The solution $(\psi_{m_1}^k, \psi_{m_2}^k, \psi_f^k)$ to the time-discrete problem in Definition \ref{def:sol_discrete_problem} satisfies
\begin{equation}
\begin{aligned}
&\hspace{0.3cm} \sum_{j=1}^2 \left( \max_{l \in \lbrace 1, \ldots, N \rbrace} \int_{\Omega_{m_j}} \mathcal{W}_m(\psi_{m_j}^l) \ d\vec{x} \right) + \varepsilon^\kappa \max_{l \in \lbrace 1, \ldots, N \rbrace} \int_{\Omega_f} \mathcal{W}_f(\psi_f^l) \ d\vec{x} \\
&\hspace{0.3cm} + \frac{\Delta t \ m_K}{2} \sum_{j=1}^2 \sum_{k=1}^N \|\nabla \psi_{m_j}^k\|^2_{\Omega_{m_j}} + \varepsilon^\lambda \frac{\Delta t \ m_K}{2} \sum_{k=1}^N \|\nabla \psi_{f}^k\|^2_{\Omega_{f}} \\
&\leq \sum_{j=1}^2 \int_{\Omega_{m_j}} \mathcal{W}_m(\psi_{{m_j},I}) \ d\vec{x} + \varepsilon^\kappa \int_{\Omega_{f}} \mathcal{W}_f(\psi_{f,I}) \ d\vec{x} \\
&\hspace{0.5cm} + \frac{\Delta t \ C_{p_m}}{2 m_K} \sum_{j=1}^2 \sum_{k=1}^N \|f_{m_j}^k\|^2_{\Omega_{m_j}} + \varepsilon^{-\lambda} \frac{\Delta t \ C_{p_f}}{2 m_K} \sum_{k=1}^N \|f_{f}^k\|^2_{\Omega_{f}}.
\end{aligned}
\end{equation}
\begin{proof}
We test in \eqref{eq:discrete_solution} with the triple $(\phi_{m_1}, \phi_{m_2}, \phi_f) = (\psi_{m_1}^k, \psi_{m_2}^k, \psi_f^k)$, which yields
\begin{equation}
\begin{aligned}
&\hspace{0.3cm} \sum_{j=1}^2 (S_m(\psi_{m_j}^k)-S_m(\psi_{m_j}^{k-1}), \psi_{m_j}^k)_{\Omega_{m_j}} + \varepsilon^\kappa (S_f(\psi_{f}^k)-S_f(\psi_{f}^{k-1}), \psi_{f}^k)_{\Omega_{f}} \\
&\hspace{0.3cm}+ \Delta t \sum_{j=1}^2 (K_m(S_m(\psi_{m_j}^k))\nabla \psi_{m_j}^k, \nabla \psi_{m_j}^k)_{\Omega_{m_j}} + \varepsilon^\lambda \Delta t (K_f(S_f(\psi_{f}^k))\nabla \psi_{f}^k, \nabla \psi_{f}^k)_{\Omega_{f}} \\
&= \sum_{j=1}^2 \Delta t (f_{m_j}^k, \psi_{m_j}^k)_{\Omega_{m_j}} + \Delta t (f_{f}^k, \psi_{f}^k)_{\Omega_{f}}.
\end{aligned}
\end{equation}
Poincar\'{e}'s inequality gives
\begin{equation}
\|\nabla \psi_\rho^k\|_{\Omega_\rho}^2 \geq \frac{\|\nabla \psi_\rho^k\|_{\Omega_\rho}^2}{2} + \frac{\|\psi_\rho^k\|_{\Omega_\rho}^2}{2 C_{p_\rho}},
\end{equation}
for $\rho \in \lbrace m_1, m_2, f \rbrace$, where $C_{p_\rho} > 0$ denotes the Poincar\'{e} constant of the respective subdomain. 
The geometries of $\Omega_{m_1}$ and $\Omega_{m_2}$ are the same, and so are the Poincar\'{e} constants hence, for which reason we set $C_{p_m} := C_{p_{m_1}} = C_{p_{m_2}}$, but note that for more general geometries, one can simply set $C_{p_m} := \max \lbrace C_{p_{m_1}}, C_{p_{m_2}} \rbrace$ in the following estimates. \\
Making use of this together with Assumption $(A_K)$ and equation \eqref{eq:W_properties}$_2$ in Lemma \ref{lemma:W_properties}, we estimate
\begin{equation}
\begin{aligned}
&\hspace{0.0cm} \sum_{j=1}^2 \int_{\Omega_{m_j}} \mathcal{W}_m(\psi_{m_j}^k) \ d\vec{x} + \varepsilon^\kappa \int_{\Omega_{f}} \mathcal{W}_f(\psi_{f}^k) \ d\vec{x} + \frac{\Delta t \ m_K}{2} \sum_{j=1}^2 \|\nabla \psi_{m_j}^k\|^2_{\Omega_{m_j}} \\
&+ \frac{\Delta t \ m_K}{2} \varepsilon^\lambda \|\nabla \psi_{f}^k\|^2_{\Omega_{f}}  
+ \frac{\Delta t \ m_K}{2 C_{p_m}} \sum_{j=1}^2 \|\psi_{m_j}^k\|^2_{\Omega_{m_j}} + \varepsilon^\lambda \frac{\Delta t \ m_K \|\psi_{f}^k\|^2_{\Omega_{f}}}{2 C_{p_f}} \\
&\leq \sum_{j=1}^2 \int_{\Omega_{m_j}} \mathcal{W}_m(\psi_{m_j}^{k-1}) \ d\vec{x} + \varepsilon^\kappa \int_{\Omega_{f}} \mathcal{W}_f(\psi_{f}^{k-1}) \ d\vec{x} + \frac{\Delta t \ C_{p_m}}{2 m_K} \sum_{j=1}^2 \|f_{m_j}^k\|^2_{\Omega_{m_j}} +\\ &\hspace{0.1cm}  \varepsilon^{-\lambda} \frac{\Delta t \ C_{p_f}}{2 m_K} \|f_{f}^k\|^2_{\Omega_{f}} 
+ \frac{\Delta t \ m_K}{2 C_{p_m}} \sum_{j=1}^2  \|\psi_{m_j}^k\|^2_{\Omega_{m_j}} + \varepsilon^\lambda \frac{\Delta t \ m_K \|\psi_{f}^k\|^2_{\Omega_{f}}}{2 C_{p_f}},
\end{aligned}
\end{equation}
where we applied the Cauchy--Schwarz inequality and Young's inequality. Summing over $k$ from $1$ to $l$ for an arbitrary $1 \leq l \leq N$ leaves us with
\begin{equation}
\begin{aligned}
&\hspace{0.0cm} \sum_{j=1}^2 \int_{\Omega_{m_j}} \mathcal{W}_m(\psi_{m_1}^l) \ d\vec{x} + \varepsilon^\kappa \int_{\Omega_f} \mathcal{W}_f(\psi_f^l) \ d\vec{x} + \frac{\Delta t \ m_K}{2} \sum_{j=1}^2 \sum_{k=1}^l \|\nabla \psi_{m_j}^k\|^2_{\Omega_{m_j}}  \\
& + \varepsilon^\lambda \frac{\Delta t \ m_K}{2} \sum_{k=1}^l \|\nabla \psi_{f}^k\|^2_{\Omega_{f}} 
\leq \sum_{j=1}^2 \int_{\Omega_{m_j}} \mathcal{W}_m(\psi_{{m_j},I}) \ d\vec{x} + \varepsilon^\kappa \int_{\Omega_{f}} \mathcal{W}_f(\psi_{f,I}) \ d\vec{x}  \\
& + \frac{\Delta t \ C_{p_m}}{2 m_K} \sum_{j=1}^2 \sum_{k=1}^l \|f_{m_j}^k\|^2_{\Omega_{m_j}} + \varepsilon^{-\lambda} \frac{\Delta t \ C_{p_f}}{2 m_K}\sum_{k=1}^l \|f_{f}^k\|^2_{\Omega_{f}},
\end{aligned}
\end{equation}
which finishes the proof.
\end{proof}
\begin{remark}[Non-degenerate case]
\label{rmrk:non-degenerate}
In the strictly parabolic case as considered in this work, where an $m_S > 0$ exists such that $0 < m_S \leq S_\rho'(\psi_\rho)$ for all $\psi_\rho \in \mathbb{R}$, we immediately obtain an $L^2$ bound for $\psi_\rho^k$ from the first two terms in Lemma \ref{lemma:a_priori_estimate_discrete} by Lemma \ref{lemma:W_properties}:
\begin{equation}
\begin{aligned}
m_S \frac{\|\psi_{m_j}^k\|_{\Omega_{m_j}}^2}{2} \leq \int_{\Omega_{m_j}} \mathcal{W}_m(\psi_{m_j}^k) \ d\vec{x}, \quad \text{and} \quad m_S \frac{\|\psi_{f}^k\|_{\Omega_{f}}^2}{2} \leq \int_{\Omega_{f}} \mathcal{W}_f(\psi_{f}^k) \ d\vec{x}.
\end{aligned}
\end{equation}
\end{remark}
\end{lemma}
In what follows, we prove the $L^\infty$ stability of the time-discrete solution. 
We define the non-negative and non-positive cut of a function $u \in W^{1,2}(\Omega)$ by
\begin{equation}
\pcut{u} := \max\lbrace u, 0 \rbrace, \qquad \ncut{u} := \min\lbrace u, 0 \rbrace.
\end{equation}
Note that $[u]_+, [u]_- \in W^{1,2}(\Omega)$, see e.g. \cite[Lemma 7.6]{GilbargTrudinger}. 
\begin{lemma}[{\it A priori} estimate II]
\label{lemma:l_infty_estimate_discrete}
For each $\Delta t > 0$, $\rho \in \lbrace m_1, m_2, f \rbrace$, and $k \in \lbrace 1, \ldots, N \rbrace$, it holds
\begin{equation}
\|\psi_{\rho}^k\|_{L^\infty(\Omega_{\rho})} \leq M_\psi \left(k \Delta t + 1\right),
\end{equation}
where
\begin{equation}
M_\psi := \max \left\lbrace M_\rho, \frac{M_f}{m_S} \right\rbrace.
\end{equation}
\begin{proof}
The proof is done by induction. For $k = 0$, the statement holds due to Assumption $(A_\rho)$. Assume now that $\|\psi_\rho^{k-1}\|_{L^\infty(\Omega_\rho)} < M_\psi \left((k-1) \Delta t + 1\right).$ First, we show that $\psi_\rho^{k} \leq M_\psi \left(k \Delta t + 1\right)$ almost everywhere in $\Omega_\rho$. \\
We test equation \eqref{eq:discrete_solution} with $\phi_{\rho} = \pcut{\psi_\rho^k - M_\psi(k \Delta t + 1)}$. These test functions satisfy the required transmission condition because $\psi_{m_j}^k = \psi_f^k$ on $\Gamma_j$. Adding some terms on both sides of the equation, we obtain
\begin{equation}
\begin{aligned}
&\hspace{0.0cm}\sum_{j=1}^2 \left(S_m(\psi_{m_j}^k)-S_m\left(M_\psi(k \Delta t + 1)\right), \pcut{\psi_{m_j}^k - M_\psi(k \Delta t + 1)}\right)_{\Omega_{m_j}} \\
&\hspace{0.0cm} + \varepsilon^\kappa \left(S_f(\psi_{f}^k)-S_f\left(M_\psi(k \Delta t + 1)\right), \pcut{\psi_f^k - M_\psi(k \Delta t + 1)}\right)_{\Omega_{f}} \\
&\hspace{0.0cm} + \Delta t \sum_{j=1}^2 \left(K_m(S_m(\psi_{m_j}^k)) \nabla \left(\psi_{m_j}^k - M_\psi(k \Delta t + 1)\right), \nabla \pcut{\psi_{m_j}^k - M_\psi(k \Delta t + 1)}\right)_{\Omega_{m_j}} \\
&\hspace{0.0cm} + \varepsilon^\lambda \Delta t \left(K_f(S_f(\psi_f^k)) \nabla \left(\psi_{f}^k - M_\psi(k \Delta t + 1)\right), \nabla \pcut{\psi_f^k - M_\psi(k \Delta t + 1)}\right)_{\Omega_{f}} \\
&=\sum_{j=1}^2 \left(S_m(\psi_{m_j}^{k-1})-S_m\left(M_\psi(k \Delta t + 1)\right), \pcut{\psi_{m_j}^k - M_\psi(k \Delta t + 1)}\right)_{\Omega_{m_j}}\\
&\hspace{0.0cm} + \varepsilon^\kappa \left(S_f(\psi_{f}^{k-1})-S_f\left(M_\psi(k \Delta t + 1)\right), \pcut{\psi_f^k - M_\psi(k \Delta t + 1)}\right)_{\Omega_{f}} \\
&\hspace{0.0cm}+ \Delta t \sum_{j=1}^2 \left(f_{m_j}^k, \pcut{\psi_{m_j}^k - M_\psi(k \Delta t + 1)}\right)_{\Omega_{m_j}} + \Delta t \left(f_{f}^k, \pcut{\psi_f^k - M_\psi(k \Delta t + 1)}\right)_{\Omega_{f}}.
\end{aligned}
\end{equation} 	
From Assumptions $(A_S)$ and $(A_K)$, and in particular the monotonicity of $S_\rho$, we deduce
\begin{equation}
\begin{aligned}
&\hspace{0.0cm} m_S \sum_{j=1}^2 \Big{\|}\pcut{\psi_{m_j}^k - M_\psi(k \Delta t + 1)}\Big{\|}^2_{\Omega_{m_j}} + \varepsilon^\kappa m_S \Big{\|}\pcut{\psi_{f}^k - M_\psi(k \Delta t + 1)}\Big{\|}^2_{\Omega_{f}} \\
&\hspace{0.2cm} + \Delta t \ m_K \sum_{j=1}^2 \Big{\|}\nabla \pcut{\psi_{m_j}^k - M_\psi(k \Delta t + 1)}\Big{\|}^2_{\Omega_{m_j}}  \\
&\hspace{0.2cm} + \varepsilon^\lambda \Delta t \ m_K \Big{\|}\nabla \pcut{\psi_{f}^k - M_\psi(k \Delta t + 1)}\Big{\|}^2_{\Omega_{f}} \\
&\leq \sum_{j=1}^2 \left(S_m(\psi_{m_j}^{k-1})-S_m\left(M_\psi((k - 1) \Delta t + 1)\right), \pcut{\psi_{m_j}^k - M_\psi(k \Delta t + 1)}\right)_{\Omega_{m_j}}\\
&\hspace{0.2cm} + \varepsilon^\kappa \left(S_f(\psi_{f}^{k-1})-S_f\left(M_\psi((k - 1) \Delta t + 1)\right), \pcut{\psi_f^k - M_\psi(k \Delta t + 1)}\right)_{\Omega_{f}} \\
&\hspace{0.2cm}+ \Delta t \sum_{j=1}^2 \left(\left(f_{m_j}^k - m_S M_\psi\right), \pcut{\psi_{m_j}^k - M_\psi(k \Delta t + 1)}\right)_{\Omega_{m_j}} \!\! \\
&\hspace{0.2cm} + \Delta t \left(\left(f_{f}^k - m_S M_\psi\right), \pcut{\psi_f^k - M_\psi(k \Delta t + 1)}\right)_{\Omega_{f}},
\end{aligned}
\label{eq:L_infty_proof_estimate}
\end{equation}
where we used $S_\rho(\psi_\rho) \geq S_\rho(\xi_\rho) + m_S (\psi_\rho-\xi_\rho)$ on the right hand side in order to get 
\begin{equation}
S_\rho\left(M_\psi((k \Delta t + 1)\right) \geq S_\rho\left(M_\psi((k - 1) \Delta t + 1)\right) + m_S M_\psi \Delta t.
\end{equation}
Note that the first two terms on the right hand side in equation \eqref{eq:L_infty_proof_estimate} are non-positive due to the induction assumption and the monotonicity of $S_\rho$. Since $M_\psi \geq \frac{M_f}{m_S}$, the last two terms are non-positive as well, from which we infer that $\psi_\rho^k \leq M_S (k \Delta t + 1)$ almost everywhere in $\Omega_\rho$. Similarly, one tests equation \eqref{eq:discrete_solution} with $\phi_\rho = \ncut{\psi_\rho^k + M_\psi(k \Delta t + 1)}$ in order to show that $\psi_\rho^{k} \geq -M_\psi \left(k \Delta t + 1\right)$ almost everywhere in $\Omega_\rho$. This concludes the proof.
\end{proof}
\end{lemma}
\subsection{Interpolation in time}
Now, we define functions on a continuous time domain by interpolating the solutions of the time-discrete problem in time. 
We use piecewise linearly interpolated functions in addition to piecewise constant functions: for almost every $t \in (t_{k-1}, t_k]$ set
\begin{equation}
\begin{aligned}
\bar{\Psi}_{\Delta t}^\rho(t) &:= \psi_\rho^k, \\
\bar{S}_{\Delta t}^\rho(t) &:= S(\psi_\rho^k), \\
\hat{S}_{\Delta t}^\rho(t) &:= S(\psi_\rho^{k-1}) + \frac{t-t_{k-1}}{\Delta t} (S(\psi_\rho^k) - S(\psi_\rho^{k-1})).
\end{aligned}
\end{equation}
Moreover, we need the piecewise constant interpolation of the source term $\bar{f}_{\Delta t}^\rho(t) = f_\rho^k$. 
In view of the {\it a priori} estimates in Lemmas \ref{lemma:a_priori_estimate_discrete} and \ref{lemma:l_infty_estimate_discrete}, we obtain the following result for the interpolated functions:
\begin{lemma}
\label{lemma:interpolation}
The functions $\bar{\Psi}_{\Delta t}^\rho$, $\bar{S}_{\Delta t}^\rho$, and $\hat{S}_{\Delta t}^\rho$ are bounded uniformly with respect to $\Delta t$ in $L^\infty(0,T; L^2(\Omega_\rho)) \cap L^2(0,T;\mathcal{V}_\rho) \cap L^\infty(\Omega_\rho^T)$ for $\rho \in \lbrace m_1, m_2, f \rbrace$.
\end{lemma}
In order to get strong convergence in $L^2(0,T;L^2(\Omega_\rho))$, we need the following estimate for the time derivative of the saturation:
\begin{lemma}
\label{lemma:partial_t_estimate}
The functions $\hat{S}_{\Delta t}^\rho$ are uniformly bounded with respect to $\Delta t$ in 
\newline $W^{1,2}(0,T;W^{-1,2}(\Omega_\rho))$ for $\rho \in \lbrace m_1, m_2, f \rbrace$.
\begin{proof}
Since the function $\hat{S}_{\Delta t}^\rho(t)$ is piecewise linear, its weak time derivative exists, is piecewise constant, and for almost every $t \in (t_{k-1}, t_k]$ given by
\begin{equation}
\partial_t \hat{S}_{\Delta t}^\rho(t) = \frac{S_\rho(\psi_\rho^k) - S_\rho(\psi_\rho^{k-1})}{\Delta t}.
\end{equation}
We view $\partial_t \hat{S}_{\Delta t}^\rho$ as an element of $L^2(0,T;W^{-1,2}(\Omega_{\rho}))$, where $W^{-1,2}(\Omega_{\rho})$ is the dual of $W^{1,2}_0(\Omega_\rho)$ (the latter space containing the $W^{1,2}$ functions on $\Omega_\rho$ with vanishing trace on the entire boundary $\partial \Omega_\rho$). 
Testing equation \eqref{eq:discrete_solution} with arbitrary $\phi_\rho \in W^{1,2}_0(\Omega_\rho)$ and $\phi_{\sigma} \equiv 0$ for $\sigma \neq \rho$ yields the estimate
\begin{equation}
\begin{aligned}
\Big{|}\langle \partial_t \hat{S}_{\Delta t}^\rho(t), \phi_\rho \rangle_{W^{-1,2}(\Omega_\rho),W^{1,2}_0(\Omega_\rho)}\Big{|} &= \Bigg{|}\left(\frac{S_\rho(\psi_\rho^k)-S_\rho(\psi_\rho^{k-1})}{\Delta t}, \phi_\rho\right)_{\Omega_\rho}\Bigg{|} \\
&\leq \Big{|}\left(K_\rho(S_\rho(\psi_\rho^k)) \nabla \psi_\rho^k, \nabla \phi_\rho\right)_{\Omega_\rho}\Big{|} + \Big{|}\left(f_\rho^k, \phi_\rho\right)_{\Omega_\rho}\Big{|} \\
&\leq \|\phi_\rho\|_{W^{1,2}(\Omega_\rho)} \left(M_K \|\nabla \psi_\rho^k\|_{\Omega^\rho} + \|f_\rho^k\|_{\Omega_\rho}\right).
\end{aligned}
\end{equation}
Using the {\it a priori} estimate in Lemma \ref{lemma:a_priori_estimate_discrete}, we obtain
\begin{equation}
\|\partial_t \hat{S}_{\Delta t}^\rho \|_{L^2(0,T;W^{-1,2}(\Omega_\rho))} \leq C,
\label{eq:partial_t_estimate}
\end{equation}
which finishes the proof.
\end{proof}
\end{lemma}
\begin{remark}\label{remark:time_derivative}
Note that the above estimate is independent of $\varepsilon$ for $\Omega_{m_1}, \Omega_{m_2}$. However, for $\Omega_f$, $K$ depends on $\varepsilon$ and  later we make precise the dependence of the above estimate on $\varepsilon$ and show that indeed the above estimate is independent of $\varepsilon$. 
\end{remark}

Compactness arguments give rise to the following convergent subsequences:
\begin{lemma}
There exists a $\Psi_\rho \in L^2(0,T;\mathcal{V}_\rho)$ and a subsequence $\Delta t \to 0$ along which we obtain for $\rho \in \lbrace m_1, m_2, f \rbrace$
\begin{equation}
\begin{aligned}
\lbrace \hat{S}_{\Delta t}^\rho \rbrace_{\Delta t} &\to S_\rho(\Psi_\rho) \qquad &\text{strongly in } &L^2(0,T;L^2(\Omega_\rho)), \\
\lbrace \bar{S}_{\Delta t}^\rho \rbrace_{\Delta t} &\to S_\rho(\Psi_\rho) \qquad &\text{strongly in } &L^2(0,T;L^2(\Omega_\rho)), \\
\lbrace \bar{\Psi}_{\Delta t}^\rho \rbrace_{\Delta t} &\to \Psi_\rho \qquad &\text{strongly in } &L^2(0,T;L^2(\Omega_\rho)), \\
\lbrace \bar{\Psi}_{\Delta t}^\rho \rbrace_{\Delta t} &\rightharpoonup \Psi_\rho \qquad &\text{weakly in } &L^2(0,T;\mathcal{V}_\rho).
\end{aligned}
\label{eq:convergence}
\end{equation}
\begin{proof}
The first convergence follows from the Aubin--Lions--Simon theorem \cite{Aubin,Simon} by the estimates in Lemmas \ref{lemma:a_priori_estimate_discrete} and \ref{lemma:partial_t_estimate}. The convergence of the piecewise linearly interpolated functions implies the convergence of the piecewise constantly interpolated functions towards the same limit function (see e.g. \cite[Lemma 3.2]{LenzingerSchweizer}). The third convergence is a consequence of Assumption $(A_S)$ by virtue of which the inverse function $S_\rho^{-1}$ exists and is Lipschitz continuous. Finally, the weak convergence in $L^2(0,T;\mathcal{V}_\rho)$ is provided by the Eberlein--\v{S}mulian theorem in view of the bounds in Lemma \ref{lemma:a_priori_estimate_discrete}.
\end{proof}
\end{lemma}
It remains to show that the triple of limit functions is a weak solution:
\begin{theorem}
\label{thrm:existence}
The limit $(\Psi_{m_1}, \Psi_{m_2}, \Psi_f)$ is a weak solution to Problem $\mathcal{P}_\varepsilon$ in the sense of Definition \ref{def:weak_solution}.
\begin{proof}
Let $(\phi_{m_1},\phi_{m_2},\phi_f) \in \mathcal{V}_{m_1} \times \mathcal{V}_{m_2} \times \mathcal{V}_{f}$. Summing \eqref{eq:discrete_solution} from $1$ to $k$ yields for almost every $t \in (t_{k-1},t_k)$
\begin{equation}
\begin{aligned}
&\hspace{0.3cm} \sum_{j=1}^2 \left(S_m(\bar{\Psi}_{\Delta t}^{m_j}(t)), \phi_{m_j} \right)_{\Omega_{m_j}} + \varepsilon^\kappa \left(S_f(\bar{\Psi}_{\Delta t}^{f}(t)), \phi_f \right)_{\Omega_f} - \sum_{j=1}^2 \left(S_m(\psi_{m_j,I}), \phi_{m_j} \right)_{\Omega_{m_j}}  \\
&\hspace{0.2cm} - \varepsilon^\kappa \left(S_f(\psi_{f,I}), \phi_f \right)_{\Omega_{f}} 
 + \sum_{j=1}^2 \int_0^t \left(K_m(S_m(\bar{\Psi}_{\Delta t}^{m_j}(\tau)))\nabla \bar{\Psi}_{\Delta t}^{m_j}(\tau), \nabla \phi_{m_j} \right)_{\Omega_{m_j}} \ d\tau  \\
&\hspace{0.2cm} + \varepsilon^\lambda \int_0^t \left(K_f(S_f(\bar{\Psi}_{\Delta t}^{f}(\tau)))\nabla \bar{\Psi}_{\Delta t}^{f}(\tau), \nabla \phi_{f} \right)_{\Omega_{f}} d\tau  - \sum_{j=1}^2 \int_0^t \left(\bar{f}_{m_j}(\tau), \phi_{m_j}\right)_{\Omega_{m_j}} d\tau \\
& - \int_0^t \left(\bar{f}_{f}(\tau), \phi_{f}\right)_{\Omega_{f}} d\tau  = \sum_{j=1}^2 \int_t^{t_k} \left(\bar{f}_{m_j}(\tau), \phi_{m_j}\right)_{\Omega_{m_j}} d\tau + \int_t^{t_k} \left(\bar{f}_{f}(\tau), \phi_{f}\right)_{\Omega_{f}} d\tau \\
&\hspace{3.5cm} - \sum_{j=1}^2 \int_t^{t_k} \left(K_m(S_m(\bar{\Psi}_{\Delta t}^{m_j}(\tau)))\nabla \bar{\Psi}_{\Delta t}^{m_j}(\tau), \nabla \phi_{m_j} \right)_{\Omega_{m_j}} d\tau \\
&\hspace{3.5cm}   - \varepsilon^\lambda \int_t^{t_k} \left(K_f(S_f(\bar{\Psi}_{\Delta t}^{f}(\tau)))\nabla \bar{\Psi}_{\Delta t}^{f}(\tau), \nabla \phi_{f} \right)_{\Omega_{f}} d\tau.
\end{aligned}
\label{eq:existence_proof_1}
\end{equation}

The terms on the right hand side correct the error made on the left hand side by integrating to $t$ instead of $t_k$. Now, we choose test functions $(\phi_{m_1}, \phi_{m_2}, \phi_f) \in L^2(0,T;\mathcal{V}_{m_1}) \times L^2(0,T;\mathcal{V}_{m_2}) \times L^2(0,T;\mathcal{V}_f)$ fulfilling $\phi_{m_j} = \phi_f$ on $\Gamma_j$ and integrate in time from $0$ to $T$ to get
\begin{equation}
\begin{aligned}
&\hspace{0.3cm} \sum_{j=1}^2 \int_0^T \left(S_m(\bar{\Psi}_{\Delta t}^{m_j}(t)), \phi_{m_j}(t) \right)_{\Omega_{m_j}} dt + \varepsilon^\kappa \int_0^T\left(S_f(\bar{\Psi}_{\Delta t}^{f}(t)), \phi_f(t) \right)_{\Omega_f} dt \\
&\hspace{0.2cm} - \sum_{j=1}^2 \int_0^T \left(S_m(\psi_{m_j,I}), \phi_{m_j}(t) \right)_{\Omega_{m_j}} dt - \varepsilon^\kappa \int_0^T \left(S_f(\psi_{f,I}), \phi_f(t) \right)_{\Omega_{f}} dt \\
&\hspace{0.2cm} + \sum_{j=1}^2 \int_0^T \int_0^t \left(K_m(S_m(\bar{\Psi}_{\Delta t}^{m_j}(\tau)))\nabla \bar{\Psi}_{\Delta t}^{m_j}(\tau), \nabla \phi_{m_j}(t) \right)_{\Omega_{m_j}} d\tau \ dt \\
&\hspace{0.2cm} + \varepsilon^\lambda \int_0^T \int_0^t \left(K_f(S_f(\bar{\Psi}_{\Delta t}^{f}(\tau)))\nabla \bar{\Psi}_{\Delta t}^{f}(\tau), \nabla \phi_{f}(t) \right)_{\Omega_{f}} d\tau \ dt \\
&\hspace{0.2cm} - \sum_{j=1}^2 \int_0^T \int_0^t \left(\bar{f}_{m_j}(\tau), \phi_{m_j}(t) \right)_{\Omega_{m_j}} d\tau \ dt - \int_0^T \int_0^t \left(\bar{f}_{f}(\tau), \phi_{f}(t) \right)_{\Omega_{f}} d\tau \ dt \\
&= \sum_{j=1}^2 \sum_{k=1}^N \int_{t_{k-1}}^{t_k} \int_t^{t_k} \left(\bar{f}_{m_j}(\tau), \phi_{m_j}(t) \right)_{\Omega_{m_j}} d\tau \ dt + \sum_{k=1}^N \int_{t_{k-1}}^{t_k} \int_t^{t_k} \left(\bar{f}_{f}(\tau), \phi_{f}(t) \right)_{\Omega_{f}} d\tau \ dt \\
&\hspace{0.2cm} - \sum_{j=1}^2 \sum_{k=1}^N \int_{t_{k-1}}^{t_k} \int_t^{t_k} \left(K_m(S_m(\bar{\Psi}_{\Delta t}^{m_j}(\tau)))\nabla \bar{\Psi}_{\Delta t}^{m_j}(\tau), \nabla \phi_{m_j}(t) \right)_{\Omega_{m_j}} d\tau \ dt \\
&\hspace{0.2cm} - \varepsilon^\lambda \sum_{k=1}^N \int_{t_{k-1}}^{t_k} \int_t^{t_k} \left(K_f(S_f(\bar{\Psi}_{\Delta t}^{f}(\tau)))\nabla \bar{\Psi}_{\Delta t}^{f}(\tau), \nabla \phi_{f}(t) \right)_{\Omega_{f}} d\tau \ dt.
\end{aligned}
\label{eq:existence_proof_2}
\end{equation}
From the strong $L^2$ convergence in \eqref{eq:convergence}$_2$, we infer that
\begin{equation}
\begin{aligned}
&\int_0^T \left(S_\rho(\bar{\Psi}_{\Delta t}^{\rho}(t)), \phi_{\rho}(t) \right)_{\Omega_{\rho}} dt \to \int_0^T \left(S_\rho(\Psi_{\rho}(t)), \phi_{\rho}(t) \right)_{\Omega_{\rho}} dt.
\end{aligned}
\end{equation}
Furthermore, the strong convergence of $\bar{\Psi}_{\Delta t}^{\rho}$ in $L^2(0,T;L^2(\Omega_\rho))$ in equation \eqref{eq:convergence}$_3$ and the weak convergence of the gradients in $L^2(0,T;L^2(\Omega_\rho))$ in equation \eqref{eq:convergence}$_4$ together with the Lipschitz continuity of $S_\rho$ and $K_\rho$ yield
\begin{equation}
\begin{aligned}
\int_0^T \int_0^t &\left(K_\rho(S_\rho(\bar{\Psi}_{\Delta t}^{\rho}(\tau)))\nabla \bar{\Psi}_{\Delta t}^{\rho}(\tau), \nabla \phi_{\rho}(t) \right)_{\Omega_{\rho}} d\tau \ dt  \\
\to &\int_0^T \int_0^t \left(K_\rho(S_\rho(\Psi_{\rho}(\tau)))\nabla \Psi_{\rho}(\tau), \nabla \phi_{\rho}(t) \right)_{\Omega_{\rho}} d\tau \ dt.
\end{aligned}
\end{equation}
This follows from the following considerations: due to the ellipticity  of $K_\rho$, $\bar{\Psi}_{\Delta t}^{\rho}$ strongly converges in $L^2(0,T;L^2(\Omega_\rho))$ and because of the Lipschitz continuity of $K_\rho$ and $S$, $K_\rho(S_\rho(\bar{\Psi}_{\Delta t}^{\rho}(\tau)))$ converges to $K_\rho(S_\rho(\Psi_\rho(\tau)))$ strongly in $L^2(0,T;L^2(\Omega_\rho))$. Due to boundedness of $K_\rho$ we know that there exists a $\vec{\xi}_{\rho} \in \left(L^2(\Omega_{\rho})\right)^d$ such that
\begin{equation}
K_\rho(S_\rho(\bar{\Psi}_{\Delta t}^{\rho}(\tau)))\nabla \bar{\Psi}_{\Delta t}^{\rho}(\tau) \rightharpoonup \vec{\xi}_{\rho}, \qquad \text{weakly in } \left(L^2(\Omega_{\rho})\right)^d.
\end{equation}
The identification of $\vec{\xi}_{\rho}$ to $K_\rho(S_\rho(\Psi_\rho))\nabla \Psi_\rho$ then follows by taking smoother test functions and passing to the limit. 

Moreover, the time-continuity of $f$ gives
\begin{equation}
\int_0^T \int_0^t \left(\bar{f}_{\rho}(\tau), \phi_{\rho}(t) \right)_{\Omega_{\rho}} d\tau \ dt \to \int_0^T \int_0^t \left(f_{\rho}(\tau), \phi_{\rho}(t) \right)_{\Omega_{\rho}} d\tau \ dt.
\end{equation}
In what follows, we show that the terms on the right hand side in equation \eqref{eq:existence_proof_2} vanish as $\Delta t$ approaches zero. For the terms involving a source term $f_{\rho}$, we obtain
\begin{equation}
\Bigg{|} \sum_{k=1}^N \int_{t_{k-1}}^{t_k} \int_t^{t_k} \left(\bar{f}_{\rho}(\tau), \phi_{\rho}(t) \right)_{\Omega_{\rho}} d\tau \ dt \Bigg{|} \leq \frac{(\Delta t)^2}{2} \sum_{k=1}^N \|f^k_{\rho}\|_{\Omega_{\rho}}^2 + \frac{\Delta t}{2} \int_0^T \|\phi_{\rho}\|_{\Omega_{\rho}}^2 \leq C \Delta t,
\end{equation}
where we used the Cauchy--Schwarz inequality and Young's inequality. \\
Furthermore, we get
\begin{equation}
\begin{aligned}
&\hspace{.2cm} \Bigg{|} \sum_{k=1}^N \int_{t_{k-1}}^{t_k} \int_t^{t_k} \left(K_\rho(S_\rho(\bar{\Psi}_{\Delta t}^{\rho}(\tau)))\nabla \bar{\Psi}_{\Delta t}^{\rho}(\tau), \nabla \phi_{\rho}(t) \right)_{\Omega_{\rho}} d\tau \ dt \Bigg{|} \\
&\leq \frac{(\Delta t)^2 M_K}{2} \sum_{k=1}^N \|\nabla \psi^k_{\rho}\|_{\Omega_{\rho}}^2 + \frac{\Delta t}{2} \int_0^T \|\nabla \phi_{\rho}\|_{\Omega_{\rho}}^2 \\
&\leq C \Delta t,
\end{aligned}
\end{equation}
using the {\it a priori} estimate in Lemma \ref{lemma:a_priori_estimate_discrete}. 
Therefore, in the limit $\Delta t \to 0$, we are left with
\begin{equation}
\begin{aligned}
&\hspace{0.3cm} \sum_{j=1}^2 \int_0^T \left(S_m(\Psi_{m_j}(t)), \phi_{m_j}(t) \right)_{\Omega_{m_j}} dt + \varepsilon^\kappa \int_0^T\left(S_f(\Psi_{f}(t)), \phi_f(t) \right)_{\Omega_f} dt \\
&\hspace{0.2cm} + \sum_{j=1}^2 \int_0^T \int_0^t \left(K_m(S_m(\Psi_{m_j}(\tau)))\nabla \Psi_{m_j}(\tau), \nabla \phi_{m_j}(t) \right)_{\Omega_{m_j}} d\tau \ dt \\
&\hspace{0.2cm} + \varepsilon^\lambda \int_0^T \int_0^t \left(K_f(S_f(\Psi_{f}(\tau)))\nabla \Psi_{m_j}(\tau), \nabla \phi_{f}(t) \right)_{\Omega_{f}} d\tau \ dt \\
&= \sum_{j=1}^2 \int_0^T \int_0^t \left(f_{m_j}(\tau), \phi_{m_j}(t) \right)_{\Omega_{m_j}} d\tau \ dt + \int_0^T \int_0^t \left(f_{f}(\tau), \phi_{f}(t) \right)_{\Omega_{f}} d\tau \ dt \\
&\hspace{0.2cm} + \sum_{j=1}^2 \int_0^T \left(S_m(\psi_{m_j,I}), \phi_{m_j}(t) \right)_{\Omega_{m_j}} dt + \varepsilon^\kappa \int_0^T \left(S_f(\psi_{f,I}), \phi_f(t) \right)_{\Omega_{f}} dt,
\end{aligned}
\label{eq:existence_proof_3}
\end{equation}
for all $(\phi_{m_1}, \phi_{m_2}, \phi_f) \in L^2(0,T;\mathcal{V}_{m_1}) \times L^2(0,T;\mathcal{V}_{m_2}) \times L^2(0,T;\mathcal{V}_{f})$ such that $\phi_{m_j} = \phi_f$ on $\Gamma_j$ for $j \in \lbrace 1, 2 \rbrace$. \\
Note that when choosing test functions $(\tilde{\phi}_{m_1}, \tilde{\phi}_{m_2}, \tilde{\phi}_f) \in W^{1,2}(0,T;\mathcal{V}_{m_1})$ \newline $\times W^{1,2}(0,T;\mathcal{V}_{m_2})$ $\times W^{1,2}(0,T;\mathcal{V}_{f})$ satisfying $\tilde{\phi}_{m_1}(T) = \tilde{\phi}_{m_2}(T) = \tilde{\phi}_f(T) = 0$, integration by parts yields
\begin{equation}
\begin{aligned}
& \int_0^T \int_0^t \left(f_{\rho}(\tau), \partial_t \tilde{\phi}_{\rho}(t)\right)_{\Omega_{\rho}} d\tau \ dt = - \int_0^T \left(f_{\rho}(t), \tilde{\phi}_{\rho}(t)\right)_{\Omega_{\rho}} dt, \\
 & \int_0^T \int_0^t \left(K_\rho(S_\rho(\Psi_{\rho}(\tau))) \nabla \Psi_{\rho}(\tau), \nabla \partial_t \tilde{\phi}_{\rho}(t)\right)_{\Omega_{\rho}} d\tau \ dt \\
&  = - \int_0^T \left(K_\rho(S_\rho(\Psi_{\rho}(t))) \nabla \Psi_{\rho}(t), \nabla \tilde{\phi}_{\rho}(t)\right)_{\Omega_{\rho}} dt.&
\end{aligned}
\end{equation}
Thus, selecting $\phi_\rho = \partial_t \tilde{\phi}_\rho$ in \eqref{eq:existence_proof_3} gives
\begin{equation}
\begin{aligned}
&\hspace{0.3cm} \sum_{j=1}^2 \int_0^T \left(S_m(\Psi_{m_j}(t)), \partial_t \tilde{\phi}_{m_j}(t) \right)_{\Omega_{m_j}} dt + \varepsilon^\kappa \int_0^T\left(S_f(\Psi_{f}(t)), \partial_t \tilde{\phi}_f(t) \right)_{\Omega_f} dt \\
&\hspace{0.2cm} - \sum_{j=1}^2 \int_0^T \left(K_m(S_m(\Psi_{m_j}(t)))\nabla \Psi_{m_j}(t), \nabla \tilde{\phi}_{m_j}(t) \right)_{\Omega_{m_j}} dt \\
&\hspace{0.2cm} - \varepsilon^\lambda \int_0^T \left(K_f(S_f(\Psi_{f}(t)))\nabla \Psi_{m_j}(t), \nabla \tilde{\phi}_{f}(t) \right)_{\Omega_{f}} dt \\
&= - \sum_{j=1}^2 \int_0^T \left(f_{m_j}(t), \tilde{\phi}_{m_j}(t) \right)_{\Omega_{m_j}} dt - \int_0^T \left(f_{f}(t), \tilde{\phi}_{f}(t) \right)_{\Omega_{f}} dt \\
&\hspace{0.2cm} - \sum_{j=1}^2  \left(S_m(\psi_{m_j,I}), \tilde{\phi}_{m_j}(0) \right)_{\Omega_{m_j}} - \varepsilon^\kappa \left(S_f({\psi}_{f,I}), \tilde{\phi}_f(0) \right)_{\Omega_{f}}.
\end{aligned}
\end{equation}
Therefore, equation \eqref{eq:weak_solution} holds true for all appropriate test functions. \\
In order to show that the interface conditions are satisfied, we estimate
\begin{equation}
\|\Psi_{m_j} - \Psi_f\|_{\Gamma_j^T} \leq \|\Psi_{m_j} - \bar{\Psi}_{\Delta t}^{m_j}\|_{\Gamma_j^T} + \|\bar{\Psi}_{\Delta t}^{m_j} - \bar{\Psi}_{\Delta t}^{f}\|_{\Gamma_j^T} + \|\bar{\Psi}_{\Delta t}^f - \Psi_f\|_{\Gamma_j^T}
\label{eq:existence_proof_4}
\end{equation}
and consider the terms on the right hand side individually. The second term is zero by definition of the discrete weak solution. For the first term, we obtain by the trace inequality 
\begin{equation}
\begin{aligned}
&\|\Psi_{m_j} - \bar{\Psi}_{\Delta t}^{m_j}\|^2_{\Gamma_j^T} \\
&\leq C(\Omega_{m_j}) \|\Psi_{m_j} - \bar{\Psi}_{\Delta t}^{m_j}\|_{\Omega_{m_j}^T} \left(\|\nabla \Psi_{m_j} - \nabla \bar{\Psi}_{\Delta t}^{m_j}\|_{\Omega_{m_j}^T} +  \|\Psi_{m_j} - \bar{\Psi}_{\Delta t}^{m_j}\|_{\Omega_{m_j}^T}\right).
\end{aligned}
\end{equation}
By the weak convergence in equation \eqref{eq:convergence}$_4$, the term in brackets is bounded, and from the strong convergence in equation \eqref{eq:convergence}$_3$, we get that $\|\Psi_{m_j} - \bar{\Psi}_{\Delta t}^{m_j}\|_{\Omega_{m_j}^T} \to 0$ for $\Delta t \to 0$. The third term on the right hand side of equation \eqref{eq:existence_proof_4} vanishes with an analogous argument, which finishes the proof.
\end{proof}
\end{theorem}
\begin{remark} \label{rem:partial_time_derivative}
For fixed $\varepsilon$, the estimate in the fracture can be improved by isolating the equation in the fracture. By carrying out the same procedure as above but with $\phi_f \in L^2(0,T; W_0^{1,2}(\Omega_f))$, that is, having zero boundary values on $\Omega_f$, we obtain 
\begin{equation}
\begin{aligned}
\partial_t S_f (\bar{\Psi}_{\Delta t}^f) &\rightharpoonup  \partial_t S_f(\psi_f) &\text{ weakly in } L^2(0,T; W^{-1,2}(\Omega_f)), \\
K_{f} (S_f (\bar{\Psi}_{\Delta t}^f)) \nabla \bar{\Psi}_{\Delta t}^f  &\rightharpoonup  K_{f}(S_f(\psi_f)) \nabla \psi_f  &\text{ weakly in } L^2(0,T; W_0^{1,2}(\Omega_f)),
\end{aligned}
\end{equation}
and $\psi_f$ satisfies the equation, 
\begin{equation}
\begin{aligned}
&\hspace{.3cm}  \varepsilon^\kappa \scalOfT{\partial_t S_f(\psi_f), \phi_f} + \varepsilon^\lambda \scalOfT{K_f(S_f(\psi_f))\nabla \psi_f, \nabla \phi_f} 
=  \scalOfT{f_f, \phi_f},
\label{eq:weak_solution_improved}
\end{aligned}
\end{equation}
for all $ \phi_f \in L^{2}(0,T;W^{1,2}_0(\Omega_f))$.

\end{remark}
\section{Rigorous upscaling}
\label{sec:upscaling}
In this section, we prove the convergence of Problem $\mathcal{P}_\varepsilon$ towards effective models in the limit $\varepsilon \to 0$ by means of rigorous upscaling. 
We present the upscaling for the parameter range $(\kappa, \lambda) \in [-1, \infty) \times (-\infty, 1)$ except for the case when $\kappa = -1, \lambda \in (-1,1)$. This corresponds to scenarios in which the inverse fracture width is an upper bound for the ratio of the fracture porosity to the matrix  and the fracture width is a lower bound for the ratio of the fracture hydraulic conductivity to the matrix. \\
We employ techniques from \cite{vanDuijnPop}, where upscaling was considered in the context of crystal dissolution and precipitation, and \cite{PopBogersKumar2016}, which is concerned with the upscaling of a reactive transport model. For a more detailed presentation of the results, we refer the reader to \cite{ListThesis}.
\subsection{Kirchhoff transformation and rescaling of the geometry}
We apply the Kirchhoff transform (see \cite{AltLuckhaus}) in each subdomain. For this, we introduce a function
\begin{equation}
\mathcal{K}_\rho: \mathbb{R} \to \mathbb{R}, \qquad u_\rho := \mathcal{K}_\rho(\psi_\rho) = \int_0^{\psi_\rho} K_\rho(S_\rho(\varphi)) \ d\varphi.
\end{equation}
Due to the assumptions on $K_\rho$ and $S_\rho$, the Kirchhoff transformation is invertible and one can define the property
\begin{equation}
b_\rho(u_\rho) := S_\rho \circ \mathcal{K}_\rho^{-1}(u_\rho).
\end{equation}
Note that $b_\rho$ is Lipschitz continuous due to the Lipschitz continuity of $S_\rho$ and $\mathcal{K}_\rho^{-1}$. \\
By the chain rule, one obtains $\nabla u_\rho = K_\rho(S_\rho(\psi_\rho)) \nabla \psi_\rho$, which transforms Problem $\mathcal{P}_\varepsilon$ into a semi-linear problem. Since $\mathcal{K}_\rho$ is Lipschitz continuous, the Kirchhoff transformed problem is equivalent to the original problem \cite{MarcusMizel},  and all {\it a priori} estimates from the previous section are satisfied for the Kirchhoff transformed variables, too. \\
The advantage of the Kirchhoff transformed formulation is the linear flux term. However, this comes at the cost of a non-linear transmission condition for the Kirchhoff transformed pressure variable. \\
We rescale the fracture in horizontal direction by defining $z = x/\varepsilon$, and introduce the following notations:
\begin{equation}
\begin{aligned}
\tilde{u}_f(t,z,y) &= u_f(t,z \varepsilon, y), \\
\tilde{u}_{f,I}(z,y) &= u_{f,I}(z \varepsilon, y) := \mathcal{K}_f(\psi_{f,I})(z \varepsilon, y), \\
\tilde{f}_f(t,z,y) &= f_f(t,z \varepsilon, y). \\
\end{aligned}
\end{equation}
To unify the notation, we set $z = x$ in the matrix blocks. For the notation of the domains, we use the following conventions: since the solid matrix subdomains are merely translated when $\varepsilon$ varies, we omit the $\varepsilon$ in the notation and write $\Omega_{m_j}$ without a superscript. The fracture domain will be denoted as $\Omega_f^\varepsilon := (-\frac{\varepsilon}{2},\frac{\varepsilon}{2}) \times (0,1)$, and the shorthand notation $\Omega_f := \Omega_f^1$ will be used. The solution in each domain $u_\rho^\varepsilon$ will be endowed with a superscript $\varepsilon$ to emphasise the $\varepsilon$-dependence. 
%

Moreover, we introduce the one-dimensional fracture domain for the effective models $\Gamma = \lbrace 0 \rbrace \times (0,1)$ and the function space for the fracture solution in the effective model
\begin{equation}
\bar{\mathcal{V}}_f := \lbrace u \in L^2(\Gamma): \partial_y u \in L^2(\Gamma), u = 0 \text{ on } \partial \Gamma \rbrace = W^{1,2}_0(\Gamma).
\end{equation}
Figure \ref{fig:Figure_limit} illustrates the geometry of the problem in rescaled variables and the geometry of the effective models, in which the fracture has become one-dimensional. 
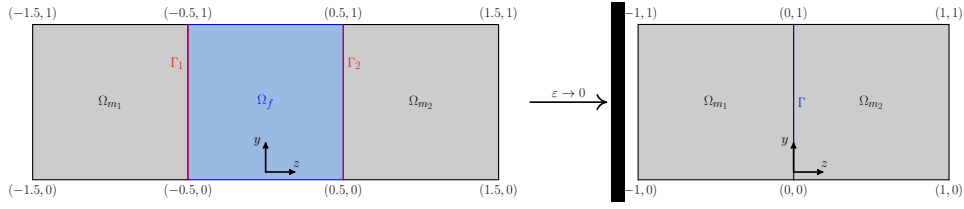
\begin{figure}[!h]
\begin{center}
\begin{minipage}{0.49\textwidth}
\resizebox{1.27\textwidth}{!}{%
\definecolor{myblue}{rgb}{0.6,0.73,0.89}
\begin{tikzpicture}[
pile/.style={ultra thick, ->, >=stealth', shorten <=-1pt, shorten
    >=2pt}]
{\Huge
\draw [line width = 2pt, fill=black!20] (-15,0) rectangle (-5,10);
\draw [line width = 2pt, fill=black!20] (5,0) rectangle (15,10);
\draw [blue, line width = 2pt, fill=myblue] (-5,0) rectangle (5,10);
\node [below](a) at (-15,0) {$(-1.5,0)$};
\node [below](b1) at (-5,0) {$(-0.5,0)$};
\node [below](b) at (-5,0) {};
\node [below](c1) at (5,0) {$(0.5,0)$};
\node [below](c) at (5,0) {};
\node [below](d) at (15,0) {$(1.5,0)$};
\node [above](e) at (-15,10) {$(-1.5,1)$};
\node [above](f1) at (-5,10) {$(-0.5,1)$};
\node [above](f) at (-5,10) {};
\node [above](g1) at (5,10) {$(0.5,1)$};
\node [above](g) at (5,10) {};
\node [above](h) at (15,10) {$(1.5,1)$};
\node(i) at (-10,5) {$\Omega_{m_1}$};
\node(j) at (10,5) {$\Omega_{m_2}$};
\node(k) [blue] at (0,5) {$\Omega_f$};
\draw [red, line width = 2pt] (b) -- (f);
\draw [red, line width = 2pt] (c) -- (g);
\node(l) [red, left] at (-5,7.5) {$\Gamma_1$};
\node(m) [red, right] at (5,7.5) {$\Gamma_2$};
\draw[pile, line width = 3pt] (0.0,0.5) -- (0.0,2.5) node [pos = 1, left] {$y$};
\draw[pile, line width = 3pt] (0.0,0.5) -- (2.0,0.5) node [pos = 1, above] {$z$};
\draw [-{>[scale=10, length=2, width=3]}, line width = 3pt] (17,5) -- (22,5) node [above, pos = 0.5] {$\varepsilon \to 0$};
}
\end{tikzpicture}
}
\end{minipage}
\begin{minipage}{0.49\textwidth}
\hspace{1.5cm}
\resizebox{0.75\textwidth}{!}{%
\definecolor{myblue}{rgb}{0.6,0.73,0.89}
\definecolor{mygray}{rgb}{0.15,0.15,0.15}
\begin{tikzpicture}[
pile/.style={ultra thick, ->, >=stealth', shorten <=-1pt, shorten
    >=2pt}]
{\Huge
\draw [line width = 2pt, fill=black!20] (-10,0) rectangle (0,10);
\draw [line width = 2pt, fill=black!20] (0,0) rectangle (10,10);
\node [below](a) at (-10,0) {$(-1,0)$};
\node [below](b) at (0,0) {$(0,0)$};
\node [below](c) at (1,0) {};
\node [below](d) at (10,0) {$(1,0)$};
\node [above](e) at (-10,10) {$(-1,1)$};
\node [above](f) at (0,10) {$(0,1)$};
\node [above](h) at (10,10) {$(1,1)$};
\node(i) at (-5,5) {$\Omega_{m_1}$};
\node(j) at (5,5) {$\Omega_{m_2}$};
\node[right](k) [blue] at (0,5) {$\Gamma$};
\draw [blue, line width = 2pt] (b) -- (f);
\draw[pile, line width = 3pt] (0.0,0.5) -- (0.0,2.5) node [pos = 1, left] {$y$};
\draw[pile, line width = 3pt] (0.0,0.5) -- (2.0,0.5) node [pos = 1, above] {$z$};
}
\end{tikzpicture}
}
\end{minipage}
\end{center}
\caption[Geometry in rescaled variables and upscaled geometry]{Geometry with two-dimensional fracture in rescaled variables (left) and upscaled geometry with one-dimensional fracture (right)}
\label{fig:Figure_limit}
\end{figure} \\
For each $\varepsilon > 0$, we define the $z$-averaged quantities
\begin{equation}
\begin{aligned}
\bar{u}_f^\varepsilon(t,y) &:=  \int_{-\frac{1}{2}}^{\frac{1}{2}} \tilde{u}_f^\varepsilon(t,z,y) \ dz,  \qquad &
\bar{f}_f(t,y) &:=  \int_{-\frac{1}{2}}^{\frac{1}{2}} \tilde{f}_f(t,z,y) \ dz,  \\
\bar{b}_f(\tilde{u}_f^\varepsilon)(t,y) &:=  \int_{-\frac{1}{2}}^{\frac{1}{2}} b_f(\tilde{u}_f^\varepsilon)(t,z,y) \ dz,  \qquad &
\bar{b}_{f}(\tilde{u}_{f,I}^\varepsilon)(y) &:=  \int_{-\frac{1}{2}}^{\frac{1}{2}} b_f(\tilde{u}_{f,I}^\varepsilon)(z,y) \ dz,
\end{aligned}
\end{equation}
and
\begin{equation}
\begin{aligned}
\breve{u}_f^\varepsilon(t) &:=  \int_0^1 \bar{u}_f^\varepsilon(t,y) \ dy, \qquad &
\breve{f}_f(t) &= \int_0^1 \bar{f}_f(t,y) \ dy, \\
\breve{b}_f(\tilde{u}_f^\varepsilon)(t) &:= \int_0^1 \bar{b}_f(\tilde{u}_f^\varepsilon)(t,y) \ dy, \qquad &
\breve{b}_{f}(\tilde{u}_{f,I}^\varepsilon) &:= \int_0^1 \bar{b}_f(\tilde{u}_{f,I})(y) \ dy.
\end{aligned}
\end{equation}
We state the weak formulation in terms of the Kirchhoff transformed and rescaled variables. In terms of the rescaled variables, the geometry is $\varepsilon$-independent and the $x$-argument of the functions associated with the fracture becomes $\varepsilon$-dependent instead:
\begin{definition}[Weak solution in rescaled variables]
\label{def:weak_sol_upscaling_variant}
A triple $(u_{m_1}^\varepsilon, u_{m_2}^\varepsilon, \tilde{u}_f^\varepsilon)$ $\in L^2(0,T;\mathcal{V}_{m_1}) \times L^2(0,T;\mathcal{V}_{m_2}) \times L^2(0,T;\mathcal{V}_f)$ is called a weak solution to the Kirchhoff transformed formulation of Problem $\mathcal{P}_\varepsilon$ if
\begin{equation}
\begin{aligned}
\mathcal{K}_m^{-1} (u_{m_1}^\varepsilon) = \mathcal{K}_f^{-1} (\tilde{u}_f^\varepsilon) \ \text{ on } \Gamma_1  \text{ and } \quad \mathcal{K}_m^{-1} (u_{m_2}^\varepsilon) = \mathcal{K}_f^{-1} (\tilde{u}_f^\varepsilon) \ \text{ on } \Gamma_2 \text{ for a.e. } t \in [0,T],
\end{aligned}
\end{equation}
in the sense of traces, and	
\begin{equation}
\begin{aligned}
&\hspace{.1cm} - \sum_{j=1}^2 \scalOmjT{b_m(u_{m_j}^\varepsilon),\partial_t \phi_{m_j}} - \varepsilon^{\kappa+1} \scalOfT{b_f(\tilde{u}_f^\varepsilon),\partial_t \phi_f} \\
&\hspace{.1cm} + \sum_{j=1}^2 \scalOmjT{\nabla u_{m_j}^\varepsilon, \nabla \phi_{m_j}} + \varepsilon^{\lambda-1} \scalOfT{\partial_z \tilde{u}_f^\varepsilon, \partial_z \phi_f} + \varepsilon^{\lambda+1} \scalOfT{\partial_y \tilde{u}_f^\varepsilon, \partial_y \phi_f} \\
&= \sum_{j=1}^2 \scalOmjT{f_{m_j}, \phi_{m_j}} + \varepsilon \scalOfT{\tilde{f}_f, \phi_f} \\
&\hspace{.1cm} + \sum_{j=1}^2 \left(b_m(u_{m_j,I}), \phi_{m_j}(0)\right)_{\Omega_{m_j}} + \varepsilon^{\kappa+1} \left(b_f(\tilde{u}_{f,I}), \phi_{f}(0)\right)_{\Omega_{f}},
\end{aligned}
\label{eq:PDE_upscaling_variant}
\end{equation}
for all $(\phi_{m_1}, \phi_{m_2}, \phi_f) \in W^{1,2}(0,T;\mathcal{V}_{m_1}) \times W^{1,2}(0,T;\mathcal{V}_{m_2}) \times W^{1,2}(0,T;\mathcal{V}_f)$ satisfying 
\begin{equation}
\phi_{m_1} = \phi_f \ \text{ on } \Gamma_1 \quad \text{and} \quad \phi_{m_2} = \phi_f \ \text{ on } \Gamma_2 \quad \text{for a.e. } t \in [0,T],
\end{equation}
and
\begin{equation}
\phi_\rho(T) = 0 \qquad \text{for } \rho \in \lbrace m_1, m_2, f \rbrace.
\end{equation}
\end{definition}
The formulation in Definition \ref{def:weak_sol_upscaling_variant} is the starting point for deriving limit models for all choices of $\kappa$ and $\lambda$. 
\subsection{Uniform estimates with respect to $\varepsilon$}
In order to prove the convergence of Problem $\mathcal{P}_\varepsilon$ towards an effective model, we establish estimates for the solution and its derivatives independent of $\varepsilon$, similar to the uniform estimates with respect to $\Delta t$ in Section \ref{sec:existence}. \\
Testing with $z$-independent functions $\phi_f(t,z,y) = \phi_f(t,y)$ in the fracture (and hence $\phi_{m_1}$ and $\phi_{m_2}$ fulfilling $\phi_{m_1}(t,z,y)|_{\Gamma_1} = \phi_f(t,y) = \phi_{m_2}(t,z,y)|_{\Gamma_2}$ for a.e. $t \in [0,T]$) in Definition \ref{def:weak_sol_upscaling_variant} gives
\begin{equation}
\begin{aligned}
&\hspace{.1cm} -\sum_{j=1}^2 \scalOmjT{b_m(u_{m_j}^\varepsilon),\partial_t \phi_{m_j}} - \varepsilon^{\kappa+1} (\bar{b}_f(\tilde{u}_f^\varepsilon),\partial_t \phi_f)_{\Gamma^T} + \sum_{j=1}^2 \scalOmjT{\nabla u_{m_j}^\varepsilon, \nabla \phi_{m_j}}  \\
&\hspace{0cm}+ \varepsilon^{\lambda+1} (\partial_y \bar{u}_f^\varepsilon, \partial_y \phi_f)_{\Gamma^T} 
= \sum_{j=1}^2 \scalOmjT{f_{m_j}, \phi_{m_j}} + \varepsilon (\bar{f}_f, \phi_f)_{\Gamma^T}\\ 
&\hspace{0cm} + \sum_{j=1}^2 \left(b_m(u_{m_j,I}), \phi_{m_j}(0)\right)_{\Omega_{m_j}} 
+ \varepsilon^{\kappa+1} \left(\bar{b}_f(\tilde{u}_{f,I}), \phi_{f}(0)\right)_{\Gamma},
\end{aligned}
\label{eq:integrated_variant}
\end{equation}
for all $(\phi_{m_1}, \phi_{m_2}, \phi_f) \in W^{1,2}(0,T;\mathcal{V}_{m_1}) \times W^{1,2}(0,T;\mathcal{V}_{m_2}) \times W^{1,2}(0,T;\bar{\mathcal{V}}_f)$ satisfying $\phi_{m_1}(t,z,y)|_{\Gamma_1} = \phi_f(t,y) = \phi_{m_2}(t,z,y)|_{\Gamma_2}$ for a.e. $t \in [0,T]$ and $\phi_\rho(T) = 0$ for $\rho \in \lbrace m_1, m_2, f \rbrace$. \\
Based on the {\it a priori} estimate in Lemma \ref{lemma:a_priori_estimate_discrete}, one shows that the solution and its derivatives can be bounded uniformly in $\varepsilon$ for $\kappa \geq -1$, $\lambda \leq -1$, and $\varepsilon > 0$ sufficiently small; in addition, we get uniform essential bounds for the solution since the constant $M_\psi$ in Lemma \ref{lemma:l_infty_estimate_discrete} is $\varepsilon$-independent:
\begin{lemma}
\label{lemma:upscaling_estimate}
There exists a $C > 0$ independent of $\varepsilon$ such that
\begin{equation}
\begin{aligned}
\sum_{j=1}^2 \|u_{m_j}^\varepsilon\|^2_{L^2(0,T;\mathcal{V}_{m_j})} + \varepsilon^{\kappa+1} \|\tilde{u}_{f}^\varepsilon\|^2_{\Omega_f^T} + \varepsilon^{\lambda+1} \|\partial_y \tilde{u}_f^\varepsilon\|^2_{\Omega_f^T} + \varepsilon^{\lambda-1} \|\partial_z \tilde{u}_f^\varepsilon\|^2_{\Omega_f^T} &\leq C, \\
\sum_{j=1}^2 \|u_{m_j}^\varepsilon\|_{L^\infty(\Omega_{m_j}^T)} + \|\tilde{u}_f^\varepsilon\|_{L^\infty(\Omega_f^T)} &\leq C.
\end{aligned}
\end{equation}
\end{lemma}
These estimates are directly carried over to the averages of fracture solution by virtue of Jensen's inequality:
\begin{lemma}
\label{lemma:upscaling_estimate_2}
There exists a $C > 0$ independent of $\varepsilon$ such that
\begin{equation}
\varepsilon^{\kappa+1} \|\bar{u}_{f}^\varepsilon\|^2_{\Gamma^T} + \varepsilon^{\lambda+1} \|\partial_y \bar{u}_f^\varepsilon\|^2_{\Gamma^T} 
+ \|\bar{u}_f^\varepsilon\|_{L^\infty(\Gamma^T)} + \varepsilon^{\kappa+1} \|\breve{u}_f^\varepsilon\|^2_{(0,T)} + \|\breve{u}_f^\varepsilon\|_{L^\infty(0,T)} \leq C.
\end{equation}
\end{lemma}
As remarked before (see Remark \ref{remark:time_derivative}), it remains to show that the time derivative estimate in the fracture $\partial_t b(u_f^\varepsilon) \in L^2(0,T; W^{-1,2}(\Omega_f))$ as obtained in Lemma \ref{lemma:partial_t_estimate} is independent of $\varepsilon$. We have the following result. 

\begin{lemma}\label{lemma:time_derivative_improved}
Under the assumption that the source term $\varepsilon^{-2 \kappa} \, \|\tilde{f}_f\|_{\Omega_f^T}^2 \leq C$ and the initial data $\varepsilon^{\lambda- \kappa - 2} \, \mathcal W(\tilde{u}^\varepsilon_{f,I})  \leq C$, the functional $\partial_t b_\rho$ is uniformly bounded with respect to $\varepsilon$ in $L^{2}(0,T;W^{-1,2}(\Omega_\rho))$ for $\rho \in \lbrace m_1, m_2, f \rbrace$.
\end{lemma}

\begin{proof}
For $\Omega_\rho, \rho = \{m_1, m_2\}$ the proof in the Lemma \ref{lemma:partial_t_estimate}, carries unchanged to show that $\partial_t b_\rho, \rho = \{m_1, m_2\}$ is bounded uniformly in $\varepsilon$ in the $L^{2}(0,T;W^{-1,2}(\Omega_\rho))$ norm. 
What remains is to consider the fracture case.  We use a duality technique for deducing the $\varepsilon$ independence.  In \eqref{eq:PDE_upscaling_variant},  choose $\phi_f \in L^2(0,T; W^{1,2}(\Omega_f)), \phi_f= 0$ on $\partial \Omega_f$ and $\phi_{m_j} =0, j \in \lbrace 1, 2 \rbrace$. Note that putting the time derivative back to $b_f$ is justified in view of Remark \ref{rem:partial_time_derivative}. This leads to, 
  
\begin{equation}
\begin{aligned}
&\hspace{.1cm}   \varepsilon^{\kappa+1} \scalOfT{\partial_t b_f(\tilde{u}_f^\varepsilon), \phi_f} 
+ \varepsilon^{\lambda-1} \scalOfT{\partial_z \tilde{u}_f^\varepsilon, \partial_z \phi_f} + \varepsilon^{\lambda+1} \scalOfT{\partial_y \tilde{u}_f^\varepsilon, \partial_y \phi_f} \\
&=  \varepsilon \scalOfT{\tilde{f}_f, \phi_f} 
\end{aligned} \label{eq:fracture_time_derivative}
\end{equation}

Our approach is to use duality technique in the above equation to deduce the $L^2(0,T;W^{-1,2}(\Omega_f))$ estimate in the fracture. However, the presence of $\varepsilon$ dependent coefficients require us to make precise this dependency. This is achieved by using the equivalency of $W^{1,2}$ norms. Let $G$ solve the following elliptic problem for a.e. $t$,

\begin{equation}
\begin{aligned}
- \nabla \cdot (A^{\varepsilon} \nabla G) &=  \partial_t b_f,  &\text{ in } &\Omega_f, \\ 
G &= 0, &\text{ on } & \partial \Omega_f,
\end{aligned}
\end{equation}
where 
\begin{align*}
A^\varepsilon = \left [
\begin{array}{lllll}
\varepsilon^{\lambda-1} & 0 \\
0 & \varepsilon^{\lambda+1}
\end{array}
\right ].
\end{align*}
For notational ease, we have suppressed dependence of $G$ on $\varepsilon$. Define an $W^{1,2}_0 (\Omega_f)$ equivalent norm, 
\begin{align*}
\| v \|_{1,\varepsilon} := \| \sqrt{A^\varepsilon} \nabla v \|_{\Omega_f},
\end{align*}
for all $v \in W^{1,2}_0(\Omega_f)$.
In terms of variational formulation, 
\begin{align*}
(\sqrt{A^\varepsilon} \nabla G, \sqrt{A^\varepsilon} \nabla v )_{\Omega_f} =  \langle \partial_t b_f, v \rangle_{W^{-1,2}(\Omega_f),W^{1,2}_0(\Omega_f)}
\end{align*}
for all $v \in W^{1,2}_0(\Omega_f)$. For the norm of the dual space to $W^{1,2}_0 (\Omega_f)$, a simple exercise gives, 
\[ \|\partial_t b_f \|_{-1,\varepsilon} =  \| \sqrt{A^\varepsilon} \nabla G \|_{\Omega_f}.\]
On the other hand, we have for the $W^{1,2}_0 (\Omega_f)$ equivalence of norms, 
\begin{align}
\label{eq:equivalentnorm} \varepsilon^{\frac{\lambda+1}{2}} \|\nabla G\|_{\Omega_f}  \leq \| \sqrt{A^\varepsilon} \nabla G \|_{\Omega_f} \leq \varepsilon^{\frac{\lambda-1}{2}} \|\nabla G\|_{\Omega_f}. 
\end{align}
Note that the $L^2$ norm for the middle term depends on the $\varepsilon$ whereas the norms on the left and right are independence of $\varepsilon$. 
With the $\varepsilon$-independent norm defined as,
\begin{align*}
\|\partial_t b_f \|_{{W^{-1,2}(\Omega_f)}} = \sup_{\phi \in W^{1,2}(\Omega_f)} \frac{{\langle \partial_t b_f, \phi \rangle_{W^{-1,2}(\Omega_f),W^{1,2}_0(\Omega_f)}}} {{ \| \nabla \phi \|_{\Omega_f}}} 
\end{align*}
 we relate using \eqref{eq:equivalentnorm}
\begin{equation*}
\begin{aligned}
\|\partial_t b_f \|_{W^{-1,2}(\Omega_f)} &= \sup_{\phi \in W^{1,2}(\Omega_f)} \frac{\langle \partial_t b_f, \phi \rangle_{W^{-1,2}(\Omega_f),W^{1,2}_0(\Omega_f)}}{ \|\nabla \phi\|_{\Omega_f}}  \\
&\leq \sup_{\phi \in W^{1,2}(\Omega_f)} \|\partial_t b_f\|_{{-1}, \varepsilon} \frac{ \|\phi\|_{1,\varepsilon}} {\| \nabla \phi \|_{\Omega_f}} \\ 
&\leq  \|\partial_t b_f\|_{{-1}, \varepsilon} \ \varepsilon^{\frac{\lambda-1}{2}}. 
\end{aligned}
\end{equation*}

%

Choosing $\phi_f = G$ in \eqref{eq:fracture_time_derivative} and integrating in time from $0$ to $t$, we obtain
\begin{equation}
\begin{aligned}
&\varepsilon^{\kappa+1}  \int_0^t \|\partial_\tau b_f \|^2_{{-1},\varepsilon} \, d\tau + \int_0^t \langle \partial_\tau b_f(\tilde{u}_f^\varepsilon), \tilde{u}_f^\varepsilon \rangle_{W^{-1,2}(\Omega_f),W^{1,2}_0(\Omega_f)} \, d\tau = \varepsilon \int_0^t \scalOf{\tilde{f}_f, G} \, d\tau\\
& \hspace{0.4cm} \leq \varepsilon^{2 + \alpha} \frac{1}{2} \|\tilde{f}_f\|^2_{\Omega_f^t}  +\frac{1}{2} \varepsilon^{-\alpha-\lambda+1} \left (\varepsilon^{\lambda-1} \|\partial_z G\|^2_{\Omega_f^t} \right), \\
\end{aligned} 
\end{equation}
where we choose $\alpha = -(\kappa+\lambda)$ to get
\begin{equation*}
\begin{aligned}
&\varepsilon^{\kappa+1}  \int_0^t \|\partial_\tau b_f \|^2_{{-1},\varepsilon} \, d\tau + \int_0^t \langle \partial_\tau b_f(\tilde{u}_f^\varepsilon), \tilde{u}_f^\varepsilon \rangle_{W^{-1,2}(\Omega_f),W^{1,2}_0(\Omega_f)} \, d\tau 
 \leq \frac{1}{2} \varepsilon^{2 - \kappa- \lambda}  \|\tilde{f}_f\|^2_{\Omega_f^t}  \\
& \hspace{5.5cm} +\frac{1}{2} \varepsilon^{\kappa+1} \int_0^t \|\partial_t b_f\|^2_{-1,\varepsilon} \, d\tau. 
\end{aligned} 
\end{equation*}
This implies,
\begin{equation}
\begin{aligned}
&\varepsilon^{\kappa+1}  \int_0^t \|\partial_\tau b_f \|^2_{{-1},\varepsilon} \, d\tau + 2 \int_0^t \langle \partial_\tau b_f(\tilde{u}_f^\varepsilon), \tilde{u}_f^\varepsilon \rangle_{W^{-1,2}(\Omega_f),W^{1,2}_0(\Omega_f)} \, d\tau \\
& \leq \varepsilon^{2 - \kappa- \lambda}  \|\tilde{f}_f\|^2_{\Omega_f^T}.
\end{aligned} \label{eq:time_derivative_eps} 
\end{equation}
In order to treat the second term in the integral above, recall $\mathcal{W}_f$  defined as
\begin{equation*}
\mathcal{W}_f(\tilde u_f) = \int_0^{\tilde u_f} b_f'(\varphi) \, \varphi \ d\varphi,
\end{equation*}
and recall the positivity of the above function as stated in Lemma \ref{lemma:W_properties}. Further, note that the second term in \eqref{eq:time_derivative_eps} can be written down as 

\begin{align*}
  \int_0^t \langle \partial_\tau b_f(\tilde{u}_f^\varepsilon), \tilde{u}_f^\varepsilon \rangle_{W^{-1,2}(\Omega_f),W^{1,2}_0(\Omega_f)} d \tau &=  \int_0^t  \dfrac{d}{d \tau} \left ( \int_{\Omega_f} \mathcal{W}_f(\tilde u_f (\tau)) \, d\vec{x} \right ) d\tau \\
 & = \int_{\Omega_f} \mathcal{W}_f(\tilde u_f (t)) \, d\vec{x} - \int_{\Omega_f} \mathcal{W}_f(\tilde u_{f, I}) \, d\vec{x}. 
\end{align*}

Using above in \eqref{eq:time_derivative_eps}, we get
\begin{equation}
\begin{aligned}
\varepsilon^{\kappa+1}  \int_0^t \|\partial_\tau b_f \|^2_{{-1},\varepsilon} \, d\tau + 2 \int_{\Omega_f} \mathcal{W}(\tilde{u}^\varepsilon_f(t)) \, d\vec{x}  &\leq \varepsilon^{2 - \kappa- \lambda}  \|\tilde{f}_f\|^2_{\Omega_f^T} + 2 \int_{\Omega_f} \mathcal W(\tilde{u}^\varepsilon_{f,I}) \, d\vec{x}.
\end{aligned} \label{eq:time_derivative_eps2}
\end{equation}
Further, we use the equivalence of norms to obtain,
\begin{equation}
\begin{aligned}
  \int_0^t \|\partial_\tau b_f \|^2_{W^{-1,2}(\Omega_f)} \, d\tau  & \leq \varepsilon^{\lambda-1} \int_0^t \|\partial_\tau b_f \|^2_{{-1},\varepsilon} \, d\tau \leq  \varepsilon^{ - 2 \kappa} \|\tilde{f}_f\|^2_{\Omega_f^t} \\
  &+ 2 \, \varepsilon^{\lambda - \kappa - 2} \int_{\Omega_f} \mathcal W(\tilde{u}^\varepsilon_{f,I}) \, d\vec{x}.
\end{aligned} \label{eq:time_derivative_eps3}
\end{equation}
As we assume that the source term $\varepsilon^{-2 \kappa} \|\tilde{f}_f\|_{\Omega_f^T}^2 \leq C$ and the initial data satisfies $\varepsilon^{\lambda- \kappa - 2} \, \mathcal W(\tilde{u}^\varepsilon_{f,I})  \leq C$, the right hand side is uniformly bounded. The positivity of the second term in \eqref{eq:time_derivative_eps} proves the lemma.  
\end{proof}

From these estimates, we obtain the following convergent subsequences using compactness arguments as $\varepsilon \to 0$:
\begin{equation}
\begin{aligned}
u_{m_j}^\varepsilon &\to U_{m_j} \qquad &\text{strongly in } & L^2(0,T;L^2(\Omega_{m_j})), \\
u_{m_j}^\varepsilon &\rightharpoonup U_{m_j} \qquad &\text{weakly in } & L^2(0,T;\mathcal{V}_{m_j}), \\
\bar{u}_f^\varepsilon &\rightharpoonup \bar{U}_f \qquad &\text{weakly in } & L^2(0,T;L^2(\Gamma)).
\end{aligned}
\label{eq:upscaling_convergence}
\end{equation}
For $\lambda \leq -1$, the gradient of the fracture solution $\nabla \tilde{u}_f^\varepsilon$ is bounded in $L^2(0,T;L^2(\Omega_f))$ according to Lemma \ref{lemma:upscaling_estimate}. From this, we infer just as for the matrix block solutions in equation \eqref{eq:upscaling_convergence} that there exists a subsequence of $\varepsilon \to 0$ along which
\begin{equation}
\begin{aligned}
\tilde{u}_f^\varepsilon &\to U_f, \qquad &\text{strongly in } & L^2(0,T;L^2(\Omega_f)), \\
\bar{u}_f^\varepsilon &\to \bar{U}_f, \qquad &\text{strongly in } & L^2(0,T;L^2(\Gamma)), \\
\bar{u}_f^\varepsilon &\rightharpoonup \bar{U}_{f} \qquad &\text{weakly in } & L^2(0,T;\bar{\mathcal{V}}_f), \\
\breve{u}_f^\varepsilon &\to \breve{U}_{f} \qquad &\text{strongly in } &L^2(0,T).
\end{aligned}
\label{eq:upscaling_convergence_2}
\end{equation}
We will make use of Proposition 4.3 in \cite{vanDuijnPop}:
\begin{proposition}
Let $\Omega = (-\frac{1}{2},\frac{1}{2}) \times (0,L)$, $f \in W^{1,2}(\Omega)$, and let $\bar{f}: [0,L] \to \mathbb{R}$ be defined as $\bar{f}(y) = \int_{-\frac{1}{2}}^{\frac{1}{2}} f(\xi,y) \ d\xi$. Then, in the sense of traces
\begin{equation}
\|f(\xi_0,\cdot) - \bar{f}\|_{(0,L)} \leq \|\partial_\xi f\|_{\Omega},
\end{equation}
for each $\xi_0 \in [-\frac{1}{2},\frac{1}{2}]$.
\end{proposition}
This proposition and Lemma \ref{lemma:upscaling_estimate} yield the following estimate:
\begin{lemma}
\label{lemma:upscaling_vanDuijn_new}
There exists a $C > 0$ independent of $\varepsilon$ such that for any $z_0 \in [-\frac{1}{2},\frac{1}{2}]$ it holds
\begin{equation}
\|\tilde{u}_f^\varepsilon(\cdot,z_0,\cdot) - \bar{u}_f^\varepsilon\|_{\Gamma^T} \leq \varepsilon^{\frac{1-\lambda}{2}} C.
\label{eq:vanDuijn}
\end{equation}
\end{lemma}
Lemma \ref{lemma:upscaling_vanDuijn_new} shows that $\lambda \leq 1$ is sufficient in order to keep the left hand side in equation \eqref{eq:vanDuijn} bounded when $\varepsilon$ vanishes (and convergence is achieved for $\lambda < 1$). \\

\subsection{Upscaling theorem}
It remains to show that the limit functions are a solution to the respective effective model:
\begin{theorem}[Upscaling theorem]
For the following ranges of $\kappa$ and $\lambda$, these tupels are a solution to the following effective models:
\begin{equation}
\begin{aligned}
& \kappa = -1,& \lambda &= -1:& \qquad & (U_{m_1}, U_{m_2}, \bar{U}_f)& &{\color{color1} \textnormal{Effective model I}}, \\
& \kappa \in (-1,\infty),& \lambda &= -1:& \qquad & (U_{m_1}, U_{m_2}, {U}_f)& &{\color{color2}\textnormal{Effective model II}},\\
& \kappa = -1, & \lambda &\in (-\infty, -1):& \qquad & (U_{m_1}, U_{m_2}, \breve{U}_f) & &{\color{color3}\textnormal{Effective model III}}, \\
& \kappa \in (-1,\infty), & \lambda &\in (-\infty, -1):& \qquad & (U_{m_1}, U_{m_2}, \breve{U}_f) & & {\color{color4}\textnormal{Effective model IV}},\\
& \kappa \in (-1,\infty), & \lambda &\in (-1, 1):& \qquad & (U_{m_1}, U_{m_2}, \bar{U}_f) & & {\color{color5}\textnormal{Effective model V}}.
\end{aligned}
\end{equation}
\begin{proof}
\noindent 
We give the proof for $\lambda \geq -1$, the other cases are treated similarly, where one uses spatially constant test functions for the fracture in $\eqref{eq:integrated_variant}$, i.e. $\phi_f(t,z,y) = \phi_f(t)$. Note that for the choice $\mathcal{V}_f = W^{1,2}_0(\Omega_f)$, the only  spatially constant test function is the zero function, but the extension to more general boundary conditions is straightforward. \\

Choose arbitrary test functions $(\phi_{m_1}, \phi_{m_2}, \phi_f) \in W^{1,2}(0,T;\mathcal{V}_{m_1}) \times W^{1,2}(0,T;\mathcal{V}_{m_2}) \times W^{1,2}(0,T;\bar{\mathcal{V}}_f)$ satisfying $\phi_{m_1}(t,z,y)|_{\Gamma_1} = \phi_f(t,y) = \phi_{m_2}(t,z,y)|_{\Gamma_2}$ for a.e. $t \in [0,T]$ and $\phi_\rho(T) = 0$ for $\rho \in \lbrace m_1, m_2, f \rbrace$ and denote the terms in equation \eqref{eq:integrated_variant} by $I_1, \ldots, I_{8}$. For all values of $\kappa$ and $\lambda$, the term $I_6$ vanishes in the limit $\varepsilon \to 0$ due to
\begin{equation*}
|I_6| = \varepsilon \ \Big{|}\! \left(\bar{f}_f^\varepsilon, \phi_f\right)_{\Gamma^T}\!\Big{|} \leq \varepsilon \|\bar{f}_f^\varepsilon\|_{\Gamma^T} \|\phi_f\|_{\Gamma^T} \to 0.
\end{equation*}
The terms $I_5$ and $I_7$ do not depend on $\varepsilon$ and remain unchanged as $\varepsilon$ approaches zero.
The strong $L^2$ convergence from equation \eqref{eq:upscaling_convergence}$_1$ and the Lipschitz continuity of $b_m$ give
\begin{equation*}
I_1 = - \sum_{j=1}^2 \left(b_m(u_{m_j}^\varepsilon), \partial_t \phi_{m_j}\right)_{\Omega_{m_j}^T} \to - \sum_{j=1}^2 \left(b_m(U_{m_j}), \partial_t \phi_{m_j}\right)_{\Omega_{m_j}^T}.
\end{equation*}
Making use of the weak convergence from equation \eqref{eq:upscaling_convergence}$_2$, we obtain
\begin{equation*}
I_3 = \sum_{j=1}^2 \left(\nabla u_{m_j}^\varepsilon, \nabla \phi_{m_j}\right)_{\Omega_{m_j}^T} \to \left(\nabla U_{m_j}, \nabla \phi_{m_j}\right)_{\Omega_{m_j}^T}.
\end{equation*}
As regards the fracture solution, we distinguish the cases $\lambda = -1$ and $\lambda > -1$: in case of $\lambda = -1$, the weak convergence from equation \eqref{eq:upscaling_convergence_2}$_3$ yields
\begin{equation*}
I_4 = \left(\partial_y \bar{u}_f^\varepsilon, \partial_y \phi_f\right)_{\Gamma^T} \to \left(\partial_y \bar{U}_f, \partial_y \phi_f\right)_{\Gamma^T},
\end{equation*}
whereas for $\lambda > -1$, we get
\begin{equation*}
|I_4| = \varepsilon^{\lambda+1} \ \Big{|}\!\left(\partial_y \bar{u}_f^\varepsilon, \partial_y \phi_f\right)_{\Gamma^T}\!\Big{|} \leq \varepsilon^{\lambda+1} \|\partial_y \bar{u}_f^\varepsilon\|_{\Gamma^T} \|\partial_y \phi_f\|_{\Gamma^T} \leq \varepsilon^{\frac{\lambda+1}{2}} C \|\partial_y \phi_f\|_{\Gamma^T} \to 0,
\end{equation*}
where we made use of the estimate for $\partial_y \bar{u}_f^\varepsilon$ in Lemma \ref{lemma:upscaling_estimate_2}. \\
It remains to consider the terms $I_2$ and $I_{8}$. Here, we make a distinction between the cases $\kappa = -1$ and $\kappa > -1$. First, consider the case $\kappa = -1$, where $I_2 = -\left(\bar{b}_f(\tilde{u}_f^\varepsilon), \partial_t \phi_f\right)_{\Gamma^T}$ and $I_{8}$ is independent of $\varepsilon$.  \\
For the term $I_2$, we start by estimating
\begin{equation*}
\begin{aligned}
\Bigg{|}\left(\bar{b}_f(\tilde{u}_f^\varepsilon) - b_f(\bar{U}_f), \partial_t \phi_f\right)_{\Gamma^T}\Bigg{|} &\leq \Bigg{|}\left(\bar{b}_f(\tilde{u}_f^\varepsilon)-b_f(\bar{u}_f^\varepsilon), \partial_t \phi_f \right)_{\Gamma^T}\Bigg{|} \\
&+ \Bigg{|}\left(b_f(\bar{u}_f^\varepsilon)-b_f(\bar{U}_f), \partial_t \phi_f\right)_{\Gamma^T} \Bigg{|}.
\end{aligned}
\end{equation*}
For the first term on the right hand side we obtain using the Lipschitz continuity of $b_f$ and Lemma \ref{lemma:upscaling_vanDuijn_new}
\begin{equation*}
\begin{aligned}
\Bigg{|}\left(\bar{b}_f(\tilde{u}_f^\varepsilon)-b_f(\bar{u}_f^\varepsilon), \partial_t \phi_f \right)_{\Gamma^T}\Bigg{|} &\leq
\|\bar{b}_f(\tilde{u}_f^\varepsilon)-b_f(\bar{u}_f^\varepsilon)\|_{\Gamma^T} \|\partial_t \phi_f\|_{\Gamma^T} \\ &=
\Big{\|}\int_{-\frac{1}{2}}^{\frac{1}{2}} \left( b_f(\tilde{u}_f^\varepsilon) - b_f(\bar{u}_f^\varepsilon) \right) dz \, \Big{\|}_{\Gamma^T} \|\partial_t \phi_f\|_{\Gamma^T} \\ &\leq 
M_S \Big{\|}\int_{-\frac{1}{2}}^{\frac{1}{2}} \Big{|} \tilde{u}_f^\varepsilon - \bar{u}_f^\varepsilon \Big{|} \, dz \, \Big{\|}_{\Gamma^T} \|\partial_t \phi_f\|_{\Gamma^T} \\ &\leq
M_S C \, \varepsilon^{\frac{1-\lambda}{2}} \|\partial_t \phi_f\|_{\Gamma^T},
\end{aligned}
\end{equation*}
which goes to zero as $\varepsilon \to 0$. \\
The second term vanishes due to
\begin{equation*}
\Bigg{|}\left(b_f(\bar{u}_f^\varepsilon)-b_f(\bar{U}_f), \partial_t \phi_f\right)_{\Gamma^T} \Bigg{|} \leq M_S \|\bar{u}_f^\varepsilon - \bar{U}_f\|_{\Gamma^T} \|\partial_t \phi_f\|_{\Gamma^T},
\end{equation*}
and the strong convergence in equation \eqref{eq:upscaling_convergence_2}$_2$. This shows that
\begin{equation*}
I_2 = -\left(\bar{b}_f(\tilde{u}_f^\varepsilon), \partial_t \phi_f\right)_{\Gamma^T} \to -\left(b_f(\bar{U}_f), \partial_t \phi_f\right)_{\Gamma^T}.
\end{equation*}
For $\kappa > -1$, we estimate
\begin{equation*}
|I_2| = \varepsilon^{\kappa+1} \ \Big{|}\! \left(\bar{b}_f(\tilde{u}_f^\varepsilon),\partial_t \phi_f\right)_{\Gamma^T}\!\Big{|} \leq \varepsilon^{\kappa+1} M_S \|\tilde{u}_f^\varepsilon\|_{\Omega_f^T} \|\partial_t \phi_f\|_{\Gamma^T} \leq \varepsilon^{\kappa+1} M_S C  \|\partial_t \phi_f\|_{\Gamma^T},
\end{equation*}
which vanishes in view of Lemma \ref{lemma:upscaling_estimate_2}, and similarly, we get
\begin{equation*}
|I_{8}| = \varepsilon^{\kappa+1} \Big{|}\! \left(\bar{b}_f(\tilde{u}_{f,I}),\phi_f(0)\right)_{\Gamma}\!\Big{|} \leq  \varepsilon^{\kappa+1} M_S \|\tilde{u}_{f,I}\|_{\Omega_f} \|\phi_f(0)\|_{\Gamma} \leq \varepsilon^{\kappa+1} M_S C \|\phi_f(0)\|_{\Gamma}
\end{equation*}
vanishing in the limit. Finally, the Dirichlet interface condition for the pressure head has to be proven. It turns out that a weakly convergent subsequence in $L^2(0,T;L^2(\Gamma))$ in the fracture suffices for this purpose. 
As the weak convergence of $\bar{u}_f^\varepsilon$ towards $\bar{U}_f$ does not directly imply the weak convergence of $\mathcal{K}_f^{-1}(\bar{u}_f^\varepsilon)$ towards $\mathcal{K}_f^{-1}(\bar{U}_f)$, we define the function $\mathcal{R}(u_{m_j}) := (\mathcal{K}_f \circ \mathcal{K}_m^{-1})(u_{m_j})$ in order to transform the interface condition $\mathcal{K}_m^{-1} (U_{m_j}) = \mathcal{K}_f^{-1} (\bar{U}_f)$ on $\Gamma_j$ into a linear expression in $\bar{U}_f$, namely
\begin{equation*}
\mathcal{R}(U_{m_j}) = \bar{U}_f \qquad \text{on } \Gamma_j.
\end{equation*}
Now we take an arbitrary test function $\phi \in L^2(0,T;L^2(\Gamma_j))$ and estimate
\begin{equation*}
\begin{aligned}
\Big{|}\left(\mathcal{R}(U_{m_j}) - \bar{U}_f, \phi\right)_{\Gamma_j^T}\Big{|} &\leq \Big{|}\left(\mathcal{R}(U_{m_j}) - \mathcal{R}(u_{m_j}^\varepsilon), \phi\right)_{\Gamma_j^T}\Big{|} + \Big{|}\left(\mathcal{R}(u_{m_j}^\varepsilon) - \tilde{u}_f^\varepsilon, \phi\right)_{\Gamma_j^T}\Big{|} \\
&\hspace{0.2cm} + \Big{|}\left(\tilde{u}_f^\varepsilon - \bar{u}_f^\varepsilon, \phi\right)_{\Gamma_j^T}\Big{|} + \Big{|}\left(\bar{u}_f^\varepsilon - \bar{U}_f, \phi\right)_{\Gamma_j^T}\Big{|}.
\end{aligned}
\end{equation*}
Let us denote the terms on the right hand side by $J_1, \ldots, J_4$. As $u_{m_j}^\varepsilon$ and $\tilde{u}_f^\varepsilon$ satisfy the interface condition, we immediately get $J_2 = 0$. Note that Assumption $(A_K)$ implies the Lipschitz continuity of $\mathcal{R}$ with Lipschitz constant $\frac{M_K}{m_K}$. Making use of this and the Cauchy--Schwarz inequality yields
\begin{equation*}
J_1 \leq \frac{M_K}{m_K} \|U_{m_j} - u_{m_j}^\varepsilon\|_{\Gamma_j^T} \|\phi\|_{\Gamma_j^T},
\end{equation*}
and one shows as in the proof of Theorem \ref{thrm:existence} that $J_1 \to 0$ in the limit $\varepsilon \to 0$ using the trace inequality, the strong convergence in equation \eqref{eq:upscaling_convergence}$_1$ and the boundedness of the gradient due to the weak convergence in equation \eqref{eq:upscaling_convergence}$_2$. For the term $J_3$, we obtain
\begin{equation*}
J_3 \leq \|\tilde{u}_f^\varepsilon - \bar{u}_f^\varepsilon\|_{\Gamma_j^T} \|\phi\|_{\Gamma_j^T},
\end{equation*}
and from Lemma \ref{lemma:upscaling_vanDuijn_new} we infer that $J_3 \to 0$. Finally, the weak convergence in equation \eqref{eq:upscaling_convergence_2}$_3$ yields $J_4 \to 0$. Since $\phi \in L^2(0,T;L^2(\Gamma_j))$ was arbitrary, one has $\mathcal{R}(U_{m_j}) = \bar{U}_f$ on $\Gamma_j$ and therefore $\mathcal{K}_m^{-1} (U_{m_j}) = \mathcal{K}_f^{-1} (\bar{U}_f)$ on $\Gamma_j$ in the sense of traces, which concludes the proof.

%
 
\end{proof}
\end{theorem}

\begin{remark}
The porous matrix domain $\Omega_{m_1}$ has the  interface at $x = 0$ with fracture domain that corresponds to $z = -1$ from the fracture domain side. Similarly, for $\Omega_{m_2}$ has the  interface at $x = 0$ with fracture domain boundary at $z = 1$. In the case when the solution in the fracture domain  is independent of $z$, the fracture collapses as an interface (see the effective equations 1 -- 5 above) and the two interfaces from the porous matrix sides coincide. 
\end{remark}

\begin{remark}
We remark now the reason for leaving out the case  when $\kappa =  -1, \lambda \in (-1,1)$ in our analysis. From \eqref{eq:upscaling_convergence} we get the boundedness of $u_f^\varepsilon$ and the smoothness of   $\bar{b}_f(u_f^\varepsilon)$ implies the existence of a weak limit for $\bar{b}_f^\varepsilon$. However, the identification of this limit in \eqref{eq:integrated_variant} to $\bar b_f(u_f)$ requires a strong convergence of $\bar u_f^\varepsilon$. As the estimates in Lemma \ref{lemma:upscaling_estimate_2} show, the gradients of $\bar{u}_f^\varepsilon$ is not bounded uniformly with respect to $\varepsilon$ and therefore, the strong convergence of $\bar u_f^\varepsilon$ cannot be deduced. We therefore exclude this case in our analysis.  
\end{remark}

\section{Numerical simulation}
\label{sec:simulations}
This section is dedicated to a numerical study aiming at the numerical validation of the theoretical upscaling result. The simulation is carried out using a standard finite volume scheme implemented in MATLAB. The code solves the model in physical variables as stated in Problem $\mathcal{P}_\varepsilon$. We use a matching grid, composed of uniform rectangular cells for partitioning the two-dimensional subdomains, and intervals of equal size for the one-dimensional fracture in the effective model. The flux is computed with a two-point flux approximation (TPFA) scheme. We use an implicit Euler discretisation in time with fixed time step, and the modified Picard scheme for the linearisation. We employ a monolithic approach and solve the system of equations for the entire domain at once. 

\subsection{Realistic example ($\kappa = \lambda = -1$)}
Our numerical example deals with the injection of water into an aquifer, which is crossed by a fracture featuring a higher permeability. Boundary and initial conditions are illustrated in Figure \ref{fig:BC}. Although our analysis was limited to homogeneous Dirichlet conditions, we expect the theoretical results to hold for the more interesting boundary conditions in our numerical example as well. In the simulations with a two-dimensional fracture, the dimensionless fracture width takes the values $\varepsilon \in \lbrace 1, 0.1, 0.01, 0.001, 0.0001 \rbrace$. We impose no flow conditions on the boundary except for the inflow region in the lower edge of the left matrix block subdomain and the right upper boundary, where a Dirichlet condition allows for outflow. Thus, the water must enter or cross the fracture to leave the domain. The parameters of the van Genuchten parametrisation of the saturation and the hydraulic conductivity are listed in Figure \ref{fig:BC}. These parameters are taken from \cite{vanGenuchten}, and correspond to silt loam and Touchet silt loam in the matrix blocks and in the fracture, respectively. 
We fix the reference length $L = 1 \, [\text{m}]$ and take the end time of the simulation to be $T = 0.45$. The time step is chosen as $0.015$. The grid size is taken as $\Delta x = \Delta y = 1/160$ in the matrix blocks and $\Delta y = 1/160$ and $\Delta x \in \lbrace 1/160, 1/800, 1/4000, 1/20000, 1/100000 \rbrace$ in the fracture, corresponding to the different fracture widths $\varepsilon$. \\ 
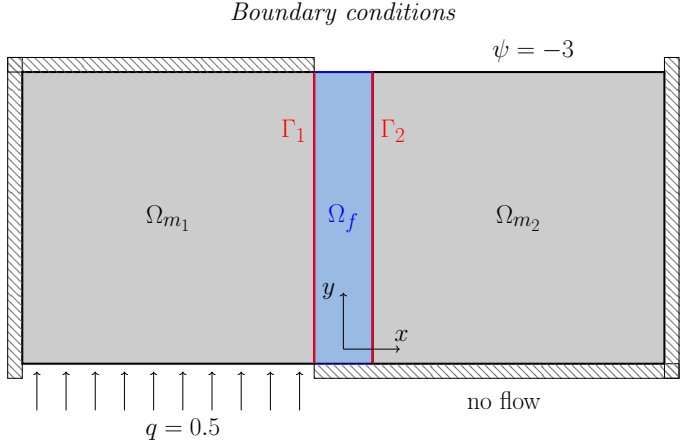
\begin{figure}[!h]
\resizebox{0.8\textwidth}{!}{%
\usetikzlibrary{patterns}
\definecolor{myblue}{rgb}{0.6,0.73,0.89}
\def\leftindent{-14}
\begin{tikzpicture}[
pile/.style={thick, ->, =>', shorten <=0pt, shorten
    >=2pt}]
{\Huge

\draw [line width = 2pt, fill=black!20] (\leftindent-11,0) rectangle (\leftindent-1,10);
\draw [line width = 2pt, fill=black!20] (\leftindent+1,0) node (v1) {} rectangle (\leftindent+11,10);
\draw [blue, line width = 2pt, fill=myblue] (\leftindent-1,0) rectangle (\leftindent+1,10);

\node(a) at (\leftindent-1,10.3) {};
\node(b) at (\leftindent-1,-0.3) {};
\node(c) at (\leftindent+1,-0.3) {};
\node(d) at (\leftindent+1,10.3) {};
\draw [red, line width = 2pt] (a) -- (b);
\draw [red, line width = 2pt] (c) -- (d);

\node(i) at (\leftindent-6,5) {$\Omega_{m_1}$};
\node(j) at (\leftindent+6,5) {$\Omega_{m_2}$};
\node(k) [blue] at (\leftindent,5) {$\Omega_f$};

\node(l) [red, left] at (\leftindent-1,8) {$\Gamma_1$};
\node(m) [red, right] at (\leftindent+1,8) {$\Gamma_2$};

\draw[pile] (\leftindent+0.0,0.5) -- (\leftindent+0.0,2.5) node [pos = 1, left] {$y$};
\draw[pile] (\leftindent+0.0,0.5) -- (\leftindent+2.0,0.5) node [pos = 1, above] {$x$};

\node(z) [] at (-7,22) {\emph{Van Genuchten parameters}};
\node(zz) [] at (\leftindent-7,22) {\emph{Geometry}};
\node(zzz) [] at (\leftindent,+12) {\emph{Boundary conditions}};

\draw [pattern=north west lines, pattern color=gray] (\leftindent-11.5,10.5) rectangle (\leftindent-11,-0.5);
\draw [pattern=north west lines, pattern color=gray] (\leftindent-1,0) rectangle (\leftindent+11.5,-0.5);
\draw [pattern=north west lines, pattern color=gray] (\leftindent-11.5,10) rectangle (\leftindent-1,10.5);
\draw [pattern=north west lines, pattern color=gray] (\leftindent+11.5,10.5) rectangle (\leftindent+11,-0.5);

\node(q) at (\leftindent-7,19) {
$
\begin{aligned} 
\Omega_{m_1} &= (-1-\varepsilon/2,-\varepsilon/2) \times (0,1) \\
\Omega_{m_2} &= (\varepsilon/2,1+\varepsilon/2) \times (0,1) \\
\Omega_{f}^\varepsilon &= (-\varepsilon/2,\varepsilon/2) \times (0,1) \\[1cm]
\end{aligned}
$
};

\node(zzzz) [] at (\leftindent-7, 16.75) {\emph{Initial condition and source term}};
\node(y) [] at (\leftindent-7, 15.5) {$\psi_I \equiv -3$};
\node(ya) [] at (\leftindent-7.12, 14.2) {$f \equiv 0$};

\node(n) [] at (\leftindent+5.5,-1.25) {no flow};
\node(o) [] at (\leftindent+6.5,10.75) {$\psi = -3$};

\foreach \x in {0,...,9}
\draw[pile] (\leftindent-10.5+\x*1,-1.6) -- (\leftindent-10.5+\x*1,-0.2);

\node [below] at (\leftindent-5.5, -1.6) {$q = 0.5$};

\node(z) at (-7, 17.5) {
\begin{tabular}{@{}lll@{}}
\toprule
                    & Fracture           & Solid matrix \\
\midrule 
$\alpha$               & $0.500$                & $0.423$ \\
$\theta_S$             & $0.469$                & $0.396$ \\
$\theta_R$             & $0.190$                & $0.131$ \\
$n$                    & $7.09$                 & $2.06$ \\
$K_S$				   & $3.507 \times 10^{-5}$ & $5.74 \times 10^{-7}$ \\
\bottomrule
\end{tabular}
};

}
\end{tikzpicture}
}
\caption{Simulation parameters for the realistic example: geometry, initial and boundary conditions, and van Genuchten parameters}
\label{fig:BC}
\end{figure}\\
The van Genuchten--Mualem parametrisation in our dimensionless setting writes as
\begin{equation}
\begin{aligned}
S_\rho(\psi_\rho) &=
\begin{cases}
\frac{\theta_{R,\rho}}{\theta_{S,\rho}} + (1 - \frac{\theta_{R,\rho}}{\theta_{S,\rho}}) \left[\frac{1}{1+(-\alpha_\rho \psi_\rho)^{n_\rho}}\right]^{\frac{{n_\rho}-1}{n_\rho}}, & \psi_\rho \leq 0, \\
\theta_{S,\rho}, & \psi_\rho > 0,
\end{cases} \\
K_\rho(S_\rho(\psi_\rho)) &= 
\begin{cases}
\Theta_{\text{eff},\rho}(\psi_\rho)^{\frac{1}{2}} \left[1-\left(1-\Theta_{\text{eff},\rho}(\psi_\rho)^{\frac{n_\rho}{n_\rho-1}}\right)^{\frac{n_\rho-1}{n_\rho}}\right]^2, & \psi_\rho \leq 0, \\
1, & \psi_\rho > 0,
\end{cases}
\end{aligned}
\label{eq:Van_Genuchten}
\end{equation}
where $\theta_S$ stands for the water content of the fully saturated porous medium, $\theta_R$ denotes the residual water content, $\alpha$ and $n$ are curve fitting parameters expressing the soil properties, and $\Theta_{\text{eff}}(\psi) := \frac{\theta(\psi) - \theta_{R}}{\theta_{S} - \theta_{R}}$ is the effective saturation. \\
The porosity ratio and the ratio of the reference hydraulic conductivities shall be given by 
\begin{equation}
\frac{\phi_f}{\phi_m} = \frac{\theta_{S,f}}{\theta_{S,m}} \varepsilon^{-1}, \qquad \text{and} \qquad
\frac{\bar{K}_f}{\bar{K}_m} = \frac{K_{S,f}}{K_{S,m}} \varepsilon^{-1},
\end{equation}
respectively, where $K_S$ denotes the saturated hydraulic conductivity.


In the limit $\varepsilon \to 0$, we expect the solution to converge towards Effective model I with the one-dimensional Richards' equation governing the flow in the fracture. \\
Figure \ref{fig:model_effective.png} depicts the pressure head and the saturation of the effective model at final time $t = T$. One observes that the pressure head in the lower left matrix block and in the lower fracture has risen, whereas the pressure head in the right matrix block has little increased. Due to the different parametrisations of the hydraulic quantities, the fracture is less saturated than the surrounding matrix blocks regardless of the pressure continuity at the fracture. \\
Figure \ref{fig:errors} shows that the averages of the pressure head across the fracture width only slightly differ from the fracture solution of the effective model, even for large $\varepsilon$. For $\varepsilon \leq 10^{-2}$, the $L^2$ errors in all subdomains lie below $10^{-4}$, and for $\varepsilon \leq 10^{-3}$, the $L^2$ errors are smaller than $10^{-5}$, which equals the tolerance of our non-linear solver. Whereas the errors in the matrix blocks are larger than the error in the fracture for wide fractures, the errors in the matrix blocks converge faster towards zero for vanishing fracture width.
\begin{figure}[!h]
\begin{center}
\includegraphics[scale=0.3]{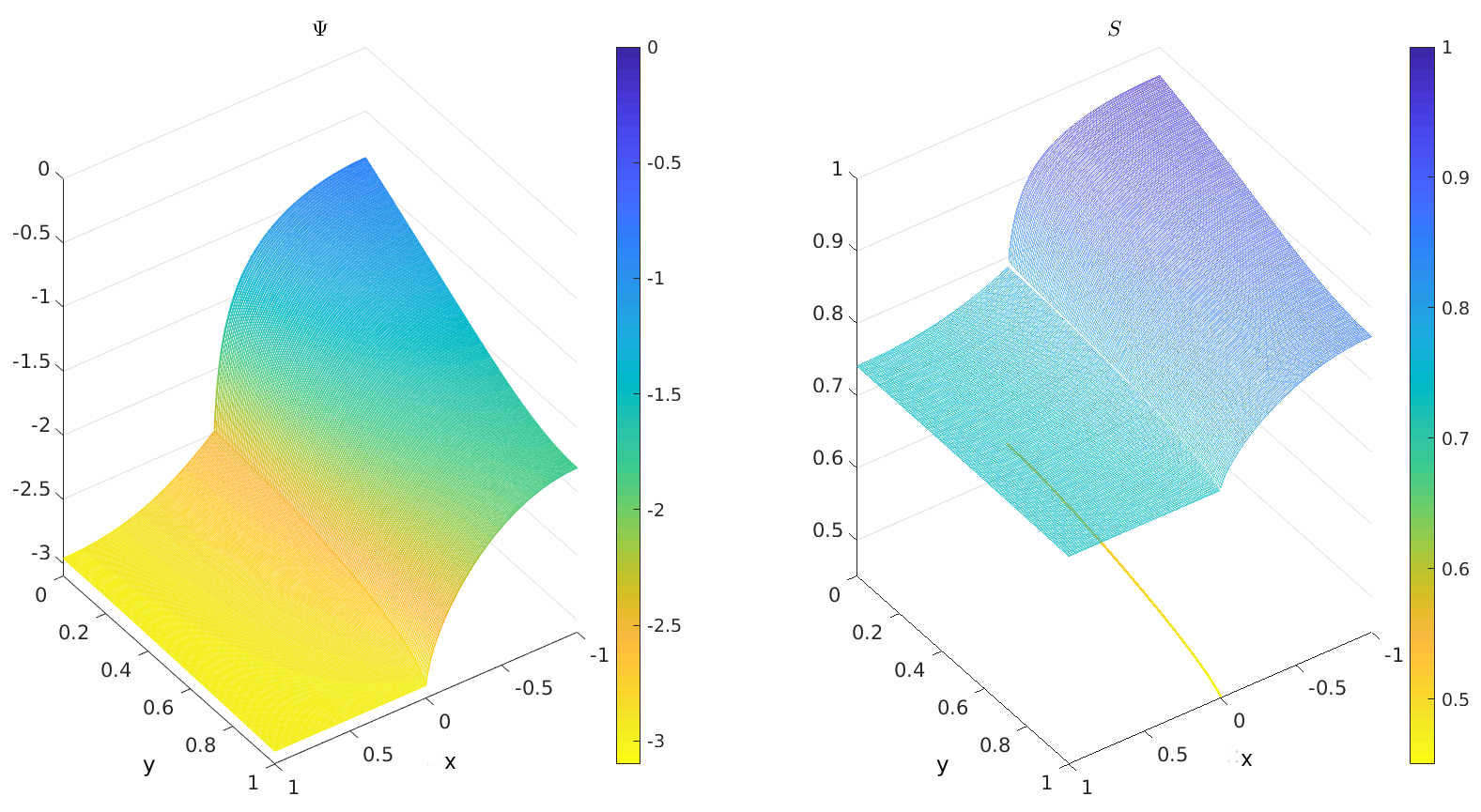}
\end{center}
\caption{Solution to the effective model at $t = 0.18$, pressure head (left) and saturation (right)}
\label{fig:model_effective.png}
\end{figure}
\begin{figure}[!h]
\begin{center}
\includegraphics[scale=0.2]{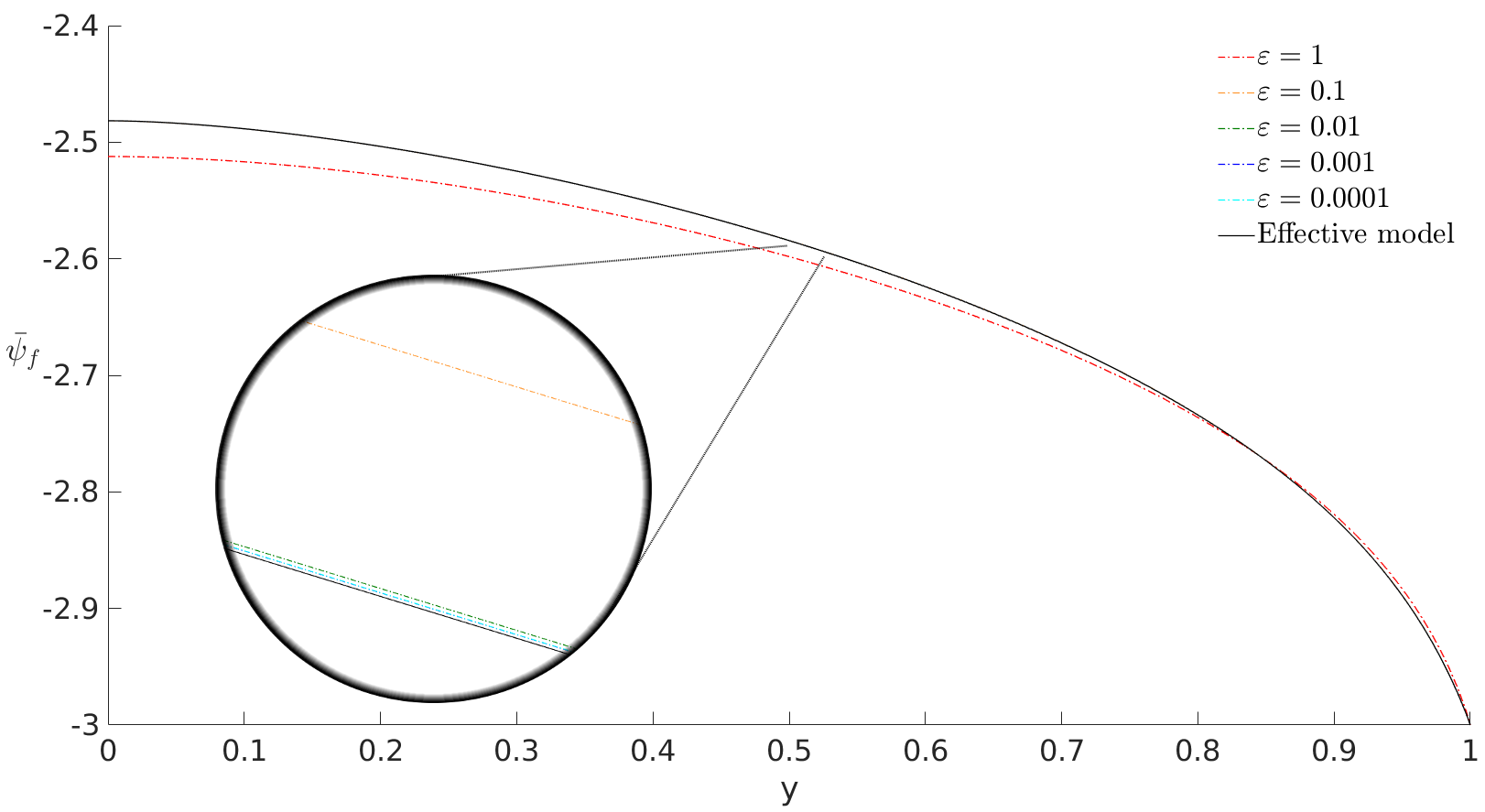}\\
\hspace{0.0cm}
\includegraphics[scale=0.2]{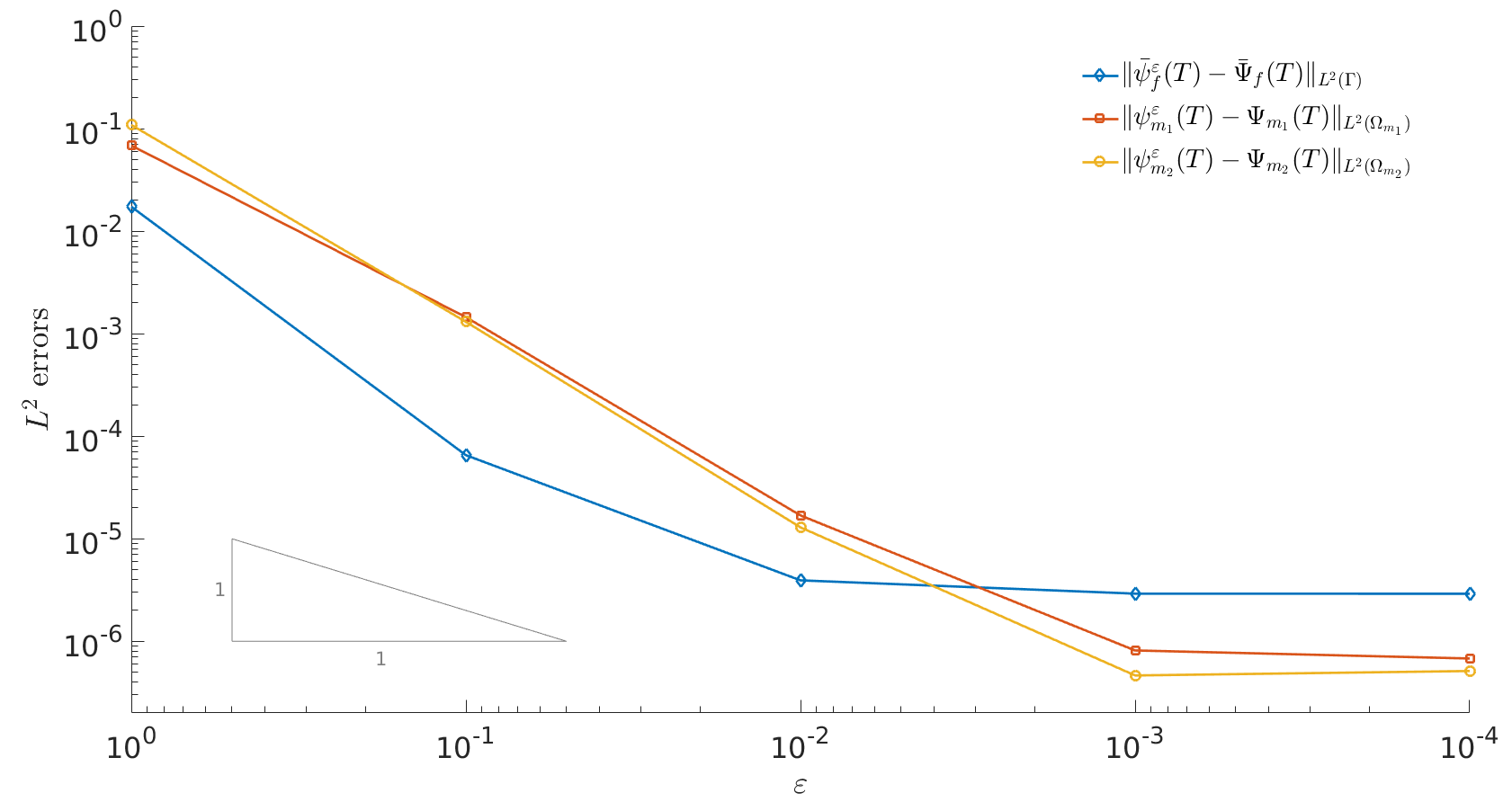}
\end{center}
\caption{$x$-averaged fracture solution along the fracture for different $\varepsilon$ and for the effective model at $t = T$ (left), $L^2$ error in the fracture and the matrix blocks for different $\varepsilon$ at $t = T$ (right)}
\label{fig:errors}
\end{figure}
\\

\section{Conclusion}
We have developed effective equations for replacing a fracture by an interface for Richards' equation. The starting geometry is a fracture of small thickness $\varepsilon$ in a porous medium. The effective models are derived as the limit of $\varepsilon \rightarrow 0$.  The ratios of porosity and absolute permeability of the fracture and the porous matrix are characterized by $\varepsilon^\kappa$ and $\varepsilon^\lambda$, respectively. The effective equations depend on the two parameters $\kappa$ and $\lambda$ and we cover the cases $\kappa > -1, \lambda < 1$. The numerical examples show that the upscaled models approximate the $\varepsilon$ problem in a satisfactory manner.  Further exploration of the numerical tests for the different upscaled models will be carried out elsewhere. 

\section{Acknowledgements}
The work of K. Kumar and F. A. Radu was partially supported by the Research Council of Norway through the projects Lab2Field no. 811716, IMMENS no. 255426, CHI  no. 25510 and Norwegian Academy of Science and Statoil through VISTA AdaSim no. 6367. I. S. Pop was supported by the Research Foundation-Flanders (FWO) through the Odysseus programme (project GOG1316N) and by Statoil through the Akademia agreement. 

\end{document}